\documentclass[12pt, notitlepage]{amsart}

\usepackage{amssymb}
\usepackage{amscd}
\usepackage{amsfonts}
\usepackage{setspace}
\usepackage{version}
\usepackage{endnotes}
\usepackage{float}
\usepackage{url}

\newtheorem{theorem}{Theorem}[section]
\newtheorem{lemma}[theorem]{Lemma}
\newtheorem{proposition}[theorem]{Proposition}

\theoremstyle{remark}
\newtheorem{remark}[theorem]{Remark}

\theoremstyle{definition}

\theoremstyle{example}

\numberwithin{equation}{section}

\newcommand{\Z}{\mathbb{Z}}
\newcommand{\R}{\mathbb{R}}

\newcommand{\C}{\mathbb{C}}
\newcommand{\E}{\mathbb{E}}

\newcommand{\e}{{\rm e}}

\newcommand{\dd}{{\rm d}}
\newcommand{\ee}{\varepsilon}
\newcommand{\1}{\mathbf{1}}

\begin{document}

\title{Vinogradov's theorem with almost equal summands}

\author{Kaisa Matom\"aki}
\address{Department of Mathematics and Statistics\\ University of Turku \\
20014 Turku \\ Finland}
\email{ksmato@utu.fi}
\thanks{KM was supported by Academy of Finland grant no. 285894.}

\author{James Maynard}
\address{Mathematical Institute\\ Radcliffe Observatory Quarter\\ Woodstock Road\\ Oxford OX2 6GG \\ United Kingdom}
\email{james.alexander.maynard@gmail.com}
\thanks{JM was supported by a Clay Research Fellowship and a Fellowship at Magdalen College, Oxford.}

\author{Xuancheng Shao}
\address{Mathematical Institute\\ Radcliffe Observatory Quarter\\ Woodstock Road\\ Oxford OX2 6GG \\ United Kingdom}
\email{Xuancheng.Shao@maths.ox.ac.uk}
\thanks{XS was supported by a Glasstone Research Fellowship.}


\maketitle

\begin{abstract}
Let $\theta > 11/20$. We prove that every sufficiently large odd integer $n$ can be written as a sum of three primes $n = p_1 + p_2 + p_3$ with $|p_i - n/3| \leq n^{\theta}$ for $i\in\{1,2,3\}$.
\end{abstract}


\section{Introduction}


Vinogradov's celebrated theorem shows that every large odd integer $N$ is the sum of three primes (see e.g. \cite[Chapter 26]{Davenport}), and was achieved by using Vinogradov's exponential sum estimates in the Hardy-Littlewood circle method. The circle method and exponential sums are both powerful general tools in analytic number theory for showing the existence of solutions to additive problems in a particular set, but it becomes increasingly difficult to apply these methods successfully as the set involved becomes increasingly sparse.

The simplest extension of Vinogradov's theorem to a sparse setting is to restrict all the primes involved to lie in a short interval $\mathcal{I}=[N/3-N^\theta,N/3+N^\theta]$. If the intervals are not too short, then one can use results about the density of zeros of the Riemann zeta function to control the exponential sums involved. By using the Huxley zero density estimate, Pan and Pan \cite{pan3} showed that if $\theta>2/3$ then every large odd integer $N$ is the sum of three primes from $\mathcal{I}$, improving on earlier work \cite{haselgrove,pan,chen,pan2}. Further progress used sieve methods to construct majorants and minorants of the indicator function of the primes which were easier to use in the exponential sum estimates. This resulted in a series of papers \cite{jia1,jia2,jia3,jia4,jia5,jia6,zhan,jia7}, the strongest being Baker and Harman's result \cite{Baker-Harman-4/7} that showed $\theta>4/7=0.5714\dots$ is sufficient, and this appears to come close to the limit of such techniques\footnote{Baker and Harman note that, whilst small improvements might be available `the value 11/20 seems out of reach without a substantial new idea'.}.

Green \cite{Green3AP} introduced a transference principle in his proof of Roth's theorem in the primes. Such a  principle provides a powerful tool to study additive problems, and for translation-invariant additive problems one can show the existence of solutions in any dense set which has a `well-controlled' majorant. In particular, essentially no lower-bound information is required beyond positive density. Vinogradov's theorem with the primes in special sets is not a translation-invariant problem (and the result does not hold for all sets of positive density), but the first and third authors \cite{MatomakiDensity,ShaoDensity,ShaoMatomaki} have shown how to use transference arguments to establish Vinogradov's theorem with certain reasonable dense subsets of the primes. In these cases one only requires the existence of elements in Bohr sets as extra lower-bound information.

In this paper we combine a transference principle with the previous work on Vinogradov's theorem with almost equal summands. The transference principle removes the need for such strong bounds for the sieve minorant occurring in the work of Baker and Harman, which then allows us to show Vinogradov's theorem with primes in shorter intervals.


\begin{theorem}\label{thm:vin-theta}
Let $\theta > 11/20$. Every sufficiently large odd integer $n$ can be written as a sum of three primes $n = p_1 + p_2 + p_3$ with $|p_i - n/3| \leq n^{\theta}$ for $i\in\{1,2,3\}$.
\end{theorem}


Our method might allow one to obtain the result with a value of $\theta$ slightly below 11/20 at the cost of significant extra complication. The best known result on primes in short intervals due to Baker, Harman and Pintz \cite{BHP01} shows the existence of primes in intervals $[x,x+x^{0.525}]$, and so one certainly cannot improve our constant by more than $0.025$ without significant new ideas which would also improve this problem. Moreover, our method requires a good numerical lower bound on the number of primes in a short interval, which is known for $\theta > 11/20$ by work of Baker, Harman and Pintz~\cite{BHP97}.

Simple probabilistic heuristics would suggest that Theorem \ref{thm:vin-theta} should hold whenever $\theta>0$, but all known arguments appear to completely fail unless $\theta>1/2$, even if one assumes strong conjectures such as the Generalized Riemann Hypothesis.

We hope that the ideas in this paper (particularly the combination of ideas from additive combinatorics with Dirichlet polynomial techniques) might have further applications - see, for example, Remark \ref{rem:Harman}.


\section{Outline}


Imagine there exists an interval $I \subset \mathbb{R}/\mathbb{Z}$ of length $1/3-\varepsilon$ and a phase $\xi\in \R/\Z$ such that every prime $p$ in $[x,x+x^\theta]$ lies in the shifted Bohr set 
\[\mathcal{B}=\{n\in\mathbb{Z}:\, \xi n \pmod{1} \in I\}.\]
Under this assumption, if $p_1,p_2,p_3\in[x,x+x^\theta]$ then we would have that $\xi(p_1+p_2+p_3)\pmod{1}$ would always lie in the sumset $I+I+I \pmod{1}$, which is an interval of length $1-3\varepsilon$ in $\mathbb{R}/\mathbb{Z}$. In particular, it would \textit{not} be the case that every large odd integer $N\in [3x+x^\theta,3x+2x^\theta]$ could be represented as the sum of three primes from $[x,x+x^{\theta}]$, since these odd integers are equidistributed modulo 1 when multiplied by $\xi$, assuming $\xi$ is in the minor arcs. Thus a necessary condition for Vinogradov's theorem with almost equal summands is that there is no such $\xi$ and $I$.

To show the existence of primes in $[x,x+x^\theta]\backslash\mathcal{B}$ for a general Bohr set $\mathcal{B}$, one would typically first construct a sieve minorant for the indicator function of the primes, and then demonstrate a positive lower bound for this sieve minorant summed over the Bohr set by Fourier expansion of the indicator function of the Bohr set and using results on the sieve minorant twisted by exponential phases. Unfortunately for short intervals it becomes difficult to obtain adequate bounds for sieve minorants twisted by exponential phases. To rule out the above possibility, however, we only need to consider sets $\mathcal{B}$ coming from an interval $I$ of length at most $1/3-\varepsilon$. Thus by inclusion-exclusion
\begin{equation}\label{eq:BohrInEx}
\#\{p\in[x,x+x^\theta]\backslash\mathcal{B}\}=\#\{p\in[x,x+x^\theta]\}-\#\{p\in[x,x+x^\theta]\cap\mathcal{B}\}.
\end{equation}
There are several results in the literature obtaining good lower bounds for the first term in \eqref{eq:BohrInEx} of the form $\alpha^-x^\theta/\log{x}$ for some constant $\alpha^-$ depending on $\theta$. Thus to show the left hand side is positive, it suffices to get a suitable upper bound for the second term. By using results on sieve majorants twisted by exponential phases, one typically obtains an upper bound of the form $\alpha^+ |I| x^\theta/\log{x}$ for some constant $\alpha^+$ depending on $\theta$. Thus we obtain a positive lower bound for \eqref{eq:BohrInEx} provided $\alpha^+<3\alpha^-$. Since it is generally easier to construct sieve majorants, this seems to be the best way to show that the left hand side is positive, and so is necessary for a proof of Vinogradov's theorem with almost equal summands in the absence of new ideas.

It turns out that such arithmetic information is actually essentially \textit{sufficient} to deduce a version of Vinogradov's theorem with almost equal summands. In particular, the main principle of proving Theorem~\ref{thm:vin-theta} is summarized in the following result. Denote by $\rho$ the indicator function for the primes, so that $\rho(n) = 1$ if $n$ is prime and $\rho(n) = 0$ otherwise.


\begin{theorem}\label{thm:vin3}
Let $\theta, \ee \in (0,1)$ and let $x$ be large. Let $\rho^+$ be a function with $\rho(n) \leq \rho^+(n)$ for every $n \in [x-x^{\theta}/3, x+x^{\theta}]$. Let $\alpha^+, \alpha^-, \eta > 0$. Assume that the following conditions hold:
\begin{enumerate}
\item Let $W = \prod_{p \leq w}p$ where $w = 0.1\log\log x$. For every residue class $c \pmod W$ with $(c,W)=1$, and every $\gamma \in \R$, we have
\[ \left|\sum_{\substack{x-x^{\theta}/3 \leq n \leq x+x^{\theta} \\ n \equiv c\pmod W}} \rho^+(n) \e(n\gamma) - \frac{\alpha^+}{\log x} \cdot \frac{W}{\varphi(W)} \sum_{\substack{x-x^{\theta}/3 \leq n \leq x+x^{\theta} \\ n \equiv c\pmod W}} \e(n\gamma)\right| \leq \frac{\eta x^{\theta}}{\varphi(W)\log x}. \]
\item For every interval $I \subset [x-x^{\theta}/3, x+x^{\theta}]$ of length $|I| \geq x^{\theta-\ee}$, and every residue class $c\pmod d$ with $(c,d) = 1$ and $d \leq \log x$, we have
\[ \sum_{\substack{n \in I \\ n\equiv c\pmod{d}}} \rho(n) \geq \frac{\alpha^- |I|}{\varphi(d)\log x}. \]
\end{enumerate}
If $\alpha^+ < 3\alpha^-$ and $\eta$ is small enough in terms of $3\alpha^- - \alpha^+$, then any odd integer in $[3x-x^{\theta}/3, 3x+x^{\theta}/3]$ can be written as a sum of three primes in $[x-x^{\theta}/3,x+x^{\theta}]$.
\end{theorem}


For $\theta = 11/20 + 2\varepsilon$, we may take $\alpha^- = 0.99$ by the work of Baker, Harman, and Pintz~\cite{BHP97}. Note that their argument is not continuous in $\theta$ at point $11/20$, so without significant numerical work it is not clear how large $\alpha^-$ one could obtain for $\theta = 11/20 - \varepsilon$. Given the result in~\cite{BHP97} our main analytic input to deduce Theorem~\ref{thm:vin-theta} from Theorem~\ref{thm:vin3}, roughly speaking, is the construction of a sieve majorant $\rho^+$, whose Fourier transform can be understood asymptotically, and which is, on average, less than three times the indicator function for the primes. This will be provided in Proposition~\ref{prop:SAz} from Section \ref{sec:rho+}.

In comparison, Baker and Harman~\cite{Baker-Harman-4/7} require an analogue of $(1)$ of Theorem \ref{thm:vin3} for a minorant for the primes as well as a majorant. Naturally the minorant is much harder to construct, and by avoiding the requirement to have such strong control when twisting by an exponential phase we can improve the exponent in the short intervals from $4/7$ to $11/20$. It is the use of the transference principle used to prove Theorem \ref{thm:vin3} which is the key to allowing us to relax such a condition to the simpler bound for the minorant given by $(2)$, and which is the ``substantial new idea" sought by Baker and Harman.


\begin{remark}\label{rem:5/9}
The first and third authors \cite[Theorem 2.3]{ShaoMatomaki} used a transference principle to show that if one has a set $\mathcal{A}$ of integers which have a suitably large intersection with \textit{all} shifted Bohr sets (with controlled complexity), and a well-controlled superset in terms of its Fourier transform, then one can deduce that all sufficiently large integers are a sum of three elements of $\mathcal{A}$. This result could be adapted to our situation of short intervals, and would then correspond to constructing a suitable sieve minorant which gave a positive lower bound to \eqref{eq:BohrInEx} for any shifted Bohr set $\mathcal{B}$ (with controlled complexity). As mentioned previously, this becomes difficult when the interval becomes short. A recent result by Harman~\cite{HarmanBeatty} should imply a lower bound of the form $\#\{p\in[x,x+x^\theta]\cap\mathcal{B}\}\gg x^\theta/\log{x}$ for $\theta > 5/9$, and thus in principle such an argument should lead to a proof of Theorem~\ref{thm:vin-theta} for $\theta > 5/9$.
\end{remark}


\begin{remark}\label{rem:Harman}
Our proof of a positive lower bound for \eqref{eq:BohrInEx} has an immediate consequence for the existence of primes in Beatty sequences in short intervals. Harman \cite{HarmanBeatty} has shown that for any fixed irrationals $\xi>1$ and $\eta$, there are primes of the form $\lfloor \xi n+\eta\rfloor$ lying in the interval $[x,x+x^\theta]$ provided $\theta>5/9$ and $x$ is sufficiently large. Since $p=\lfloor \xi n+\eta\rfloor$ is equivalent to $n+\eta-1/\xi<p/\xi\le n+\eta$, we see that it suffices to show primes $p$ in the interval $[x,x+x^\theta]$ such that $p/\xi\pmod{1}$ lies in an interval of length $1/\xi$. Thus for $1<\xi<3/2$ we obtain Harman's result in the wider range $\theta>11/20$. We thank Glyn Harman for this observation.
\end{remark}


\section{Overview of the proof of Theorem~\ref{thm:vin3}}


Since we will be doing Fourier analysis on $\Z$, it is natural to adopt the following normalization for the Fourier transform: for a finitely supported function $f: \Z \to \C$ define
\[ \widehat{f}(\gamma) = \sum_n f(n) \e(-n\gamma). \]
For two finitely supported functions $f,g: \Z \to \C$, the convolution $f*g$ is defined by
\[ f*g(n) = \sum_m f(m) g(n-m). \]
The following combinatorial statement is a transference principle from which Theorem~\ref{thm:vin3} can be deduced immediately. Here, and in the sequel, we write $[N]$ for the discrete interval $\{1,2,\cdots,N\}$.


\begin{proposition}\label{prop:kneser-sparse}
Let $\ee, \eta \in (0,1)$. Let $N$ be a positive integer and let $f_1,f_2,f_3: [N] \to \R_{\geq 0}$ be functions, with each $f \in \{f_1,f_2,f_3\}$ satisfying the following assumptions:
\begin{enumerate}
\item For each arithmetic progression $P \subset [N]$ with $|P| \geq \eta N$ we have $\E_{n \in P} f(n) \geq 1/3+\ee$;
\item There exists a majorant $\nu: [N] \to \R_{\geq 0}$ with $f \leq \nu$ pointwise, such that $\|\widehat{\nu} - \widehat{\1_{[N]}}\|_{\infty} \leq \eta N$;
\item We have $\|\widehat{f}\|_q \leq K N^{1-1/q}$ for some fixed $q,K$ with $K \geq 1$ and $2 < q < 3$.
\end{enumerate}
Then for each $n \in [N/2, N]$ we have
\[ f_1*f_2*f_3(n) \geq (c(\ee) - O_{\ee, K, q}(\eta)) N^2, \] where $c(\ee) > 0$ is a constant depending only on $\ee$.
\end{proposition}


\begin{proof}[Proof of Theorem~\ref{thm:vin3} assuming Proposition~\ref{prop:kneser-sparse}]
Let $\rho,\rho^+,\alpha^-,\alpha^+$ be as in the statement of Theorem~\ref{thm:vin3}. Let $n_0 \in [3x-x^{\theta}/3, 3x+x^{\theta}/3]$ be odd, and our goal is to represent $n_0$ as a sum of three primes in $[x-x^{\theta}/3,x+x^{\theta}]$. Let $W = \prod_{p \leq w}p$ where $w = 0.1\log\log x$, and choose $b_1,b_2,b_3\pmod W$ with $(b_i,W) = 1$ such that $b_1+b_2+b_3 \equiv n_0\pmod {W}$. Let $N = \lfloor 4x^{\theta}/(3W) \rfloor$ and $m = \lfloor (x-x^{\theta}/3)/W \rfloor$, and consider the functions $f_i,\nu_i$ defined on $[N]$ by
\[ f_i(n) = \frac{\log x}{\alpha^+} \cdot \frac{\varphi(W)}{W}\rho(W(m+n)+b_i), \ \ \nu_i(n) = \frac{\log x}{\alpha^+}\cdot \frac{\varphi(W)}{W}\rho^+(W(m+n)+b_i). \]
Then $f_i \leq \nu_i$ since $\rho \leq \rho^+$. To prove the Fourier uniformity of $\nu_i$ (condition (2)), observe that
\[ \sum_{1 \leq n \leq N} \nu_i(n) \e(n\gamma) = \e((-b_i/W-m)\gamma) \frac{\log x}{\alpha^+}  \cdot \frac{\varphi(W)}{W} \left(\sum_{\substack{x-x^{\theta}/3 \leq n \leq x+x^{\theta} \\ n \equiv b_i\pmod{W}}} \rho^+(n) \e(n\gamma/W) + O(1)\right) \]
and similarly
\[ \sum_{1 \leq n \leq N} \e(n\gamma) = \e((-b_i/W-m)\gamma) \sum_{\substack{x-x^{\theta}/3 \leq n \leq x+x^{\theta} \\ n \equiv b_i\pmod{W}}} \e(n\gamma/W) + O(1) \]
for any $\gamma \in \R$. Comparing the two equations above and using the assumption (1) about the Fourier transform of $\rho^+$, we obtain
\[ \left|\sum_{1 \leq n \leq N} \nu_i(n) \e(n\gamma) - \sum_{1 \leq n \leq N} \e(n\gamma)\right| \ll \frac{\eta}{\alpha^+} \cdot  \frac{x^{\theta}}{W}. \]
This verifies hypothesis (2) of Proposition~\ref{prop:kneser-sparse}.

On the other hand, for any arithmetic progression $P \subset [N]$ of length $\geq \eta N$, by assumption (2) we have the lower bound
\[ \sum_{n \in P} f_i(n) = \frac{\log x}{\alpha^+} \cdot \frac{\varphi(W)}{W} \sum_{n \in Q} \rho(n) \geq \frac{\alpha^-}{\alpha^+}|Q| = \frac{\alpha^-}{\alpha^+}|P|, \]
where $Q = \{W(m+n)+b: n \in P\}$ (noting that if $d$ is the common difference of $Q$ then $d/\varphi(d) = W/\varphi(W)$). Since $\alpha^+ < 3\alpha^-$, we may find $\ee > 0$ such that $\alpha^-/\alpha^+ > 1/3 + \ee$. This verifies hypothesis (1) of Proposition~\ref{prop:kneser-sparse}.

Finally, the hypothesis (3) (with $q = 5/2$ and $K = O(1)$) holds by the work of Green and Tao~\cite{GreenTaoRestr}\footnote{Precisely we need~\cite[Theorem 1.1]{GreenTaoRestr} with $F$ being the linear form $F(n) = an+b$ with its coefficients satisfying $|a|,|b| \leq N^2$ (rather than $|a|, |b| \leq N$ as stated there). This generalization follows with the same proof, since the coefficient bounds are only used to control certain divisor functions.}. 
Since $\eta$ is assumed to be small enough in terms of $\ee$, we may thus apply Proposition~\ref{prop:kneser-sparse} to find a representation of the form
\[ n = n_1 + n_2 + n_3 \]
for each $n \in [N/2, N]$, with each $n_i$ in the support of $f_i$. In particular, each $p_i = W(m+n_i)+b_i$ is a prime in the range $[x-x^{\theta}/3, x+x^{\theta}]$, and we have the representation
\[ W(3m+n) + b_1+b_2+b_3 = p_1+p_2+p_3. \]
Setting $n = (n_0 - b_1-b_2-b_3)/W - 3m$ above concludes the proof of Theorem~\ref{thm:vin3}.
\end{proof}

Via a standard transference argument~\cite{Green3AP}, Proposition~\ref{prop:kneser-sparse} is reduced to the dense case when $\nu = 1$ and thus the functions $f_i$ are bounded by $1$. We record this special case of Proposition~\ref{prop:kneser-sparse} separately:


\begin{proposition}\label{prop:kneser-dense}
For any $\ee > 0$, there exist constants $\eta = \eta(\ee) > 0$ and $c = c(\ee) > 0$ such that the following statement holds. Let $N$ be a positive integer and let $f_1,f_2,f_3: [N] \to [0,1]$ be functions, with each $f \in \{f_1,f_2,f_3\}$ satisfying the inequality
\[ \E_{n \in P} f(n) \geq 1/3 + \ee \]
for each arithmetic progression $P \subset [N]$ with $|P| \geq \eta N$. Then for each $n \in [N/2, N]$ we have
\[ f_1*f_2*f_3(n) \geq c N^2. \]
\end{proposition}


Taking $A_1,A_2$ to be the essential supports of $f_1,f_2$, respectively, up to $N/2$, we see that they have sizes larger than $(1/3) \cdot (N/2)$ by the assumption. In order for the conclusion to hold, we need $A_1+A_2$ to have size larger than $2N/3$. Thus an understanding about sets of doubling less than $4$ is required. For two subsets $A, B \subset \Z$ and $\eta > 0$, denote by $S_{\eta}(A,B)$ the set of $\eta$-popular sums in $A+B$:
\[ S_{\eta}(A,B) = \{n\in\Z: 1_A*1_B(n) \geq \eta\max(|A|,|B|)\}. \]


\begin{proposition}\label{prop:doubling4}
For any $\ee > 0$, there exists a constant $\eta = \eta(\ee) > 0$ such that the following statement holds. Let $N$ be a positive integer and let $\alpha \in [0,1/2]$. Let $A, B \subset [N]$ be two subsets with the properties that
\[ |A \cap P| \geq \alpha |P|, \ \ |B \cap P| \geq \alpha |P|, \]
for each arithmetic progression $P \subset [N]$ with $|P| \geq \eta N$. Then 
\[ |S_{\eta}(A, B)| \geq (4\alpha - \ee)N. \]
\end{proposition}


In the contrapositive, the statement roughly asserts that if a set has doubling less than $4$, then its density on a not-so-short progression is less than expected. We will deduce this from a related result of Eberhard, Green, and Manners~\cite{EGM14}, which says that if a set has doubling less than $4$, then its density on a not-so-short progression is larger than $1/2$.


\section{The transference argument}\label{sec:transfer}


In this section we prove Proposition~\ref{prop:kneser-sparse}: In Subsection~\ref{ssec:Pfdoubling4} we prove Proposition~\ref{prop:doubling4}, in Subsection~\ref{ssec:Kneser-densefromdoubling4} we deduce Proposition \ref{prop:kneser-dense} from Proposition~\ref{prop:doubling4} and in Subsection~\ref{ssec:Kneser-sparsefromKneser-dense} we deduce Proposition~\ref{prop:kneser-sparse} from Proposition~\ref{prop:kneser-dense}.

\subsection{Proof of Proposition~\ref{prop:doubling4}}
\label{ssec:Pfdoubling4}

We start by deducing Proposition~\ref{prop:doubling4} from the following:


\begin{theorem}\label{thm:doubling4}
For any $\ee > 0$, there exists a constant $\delta = \delta(\ee) > 0$ such that the following statement holds. Let $N$ be a positive integer, and let $A,B \subset [N]$ be two subsets with 
\[ |S_{\delta}(A,B)| \leq 4\min(|A|,|B|) - \ee N. \] 
Then there is an arithmetic progression $P \subset [N]$ of length $\geq \delta N$ on which $A$ has density at least $1/2 + \ee/5$.
\end{theorem}


This is an asymmetric version of~\cite[Theorem 4.1]{EGM14}; see the remarks following~\cite[Theorem 6.2]{EGM14}. To deduce Proposition~\ref{prop:doubling4}, we may clearly assume that $\ee/4 \leq \alpha \leq 1/2 - 3\ee$. We may also assume that $N$ is sufficiently large depending on $\ee$, since otherwise the assumptions imply that $A = B = [N]$ and the statement is trivial. 

Let $\delta = \delta(\ee/2) > 0$ be the constant from Theorem~\ref{thm:doubling4}. Let $q$ be the product of all positive integers up to $\delta^{-1}$. We shall show that Proposition~\ref{prop:doubling4} holds with $\eta = \ee\delta/(2q)$.

Divide $[N]$ into short intervals of length $\sim \ee\delta N$, and then divide each of these short intervals into arithmetic progressions modulo $q$. In this way we obtain a partition
\[ [N] = Q_1 \cup \cdots \cup Q_m, \]
where each $Q_i$ is an arithmetic progression of step $q$ and length $\sim \ee \delta N/q$. By assumption we have
\[ |A \cap Q_i| \geq \alpha |Q_i|, \ \ |B \cap Q_i| \geq \alpha |Q_i|, \]
for each $Q_i$. By removing elements from $A,B$ suitably, we may find $A' \subset A$ and $B' \subset B$ with the properties that
\[ |A' \cap Q_i| = \alpha|Q_i| + O(1), \ \ |B' \cap Q_i| = \alpha|Q_i| + O(1), \]
for each $Q_i$. We claim that $|S_{\delta}(A',B')| \geq (4\alpha-\ee)N$, and thus
\[ |S_{\eta}(A,B)| \geq |S_{\delta}(A,B)| \geq |S_{\delta}(A',B')| \geq (4\alpha - \ee)N, \]
as desired. To prove the claim, suppose for the purpose of contradiction that it fails.  Then by Theorem~\ref{thm:doubling4}, there exists an arithmetic progression $P \subset [N]$ of length $\geq \delta N$ on which $A'$ has density larger than $1/2$. A moment's thought reveals that $P$ is necessarily the union of some of the $Q_i$'s, up to a residual set of at most $2\ee\delta N$ elements. It follows by the construction of $A'$ that
\[ |A' \cap P| \leq \alpha |P| + O(m) + 2\ee\delta N \leq (\alpha + 3\ee) |P| \leq |P|/2, \]
a contradiction. This completes the proof of Proposition~\ref{prop:doubling4}.


\begin{remark}
Assuming that $A,B$ have densities at least $\alpha,\beta$, respectively, on not-so-short progressions, one should be able to prove using the arithmetic regularity lemma that
\[ |A+B| \geq (\min(\alpha+\beta,1) - \ee)2N. \]
However, we do not see how to directly deduce this asymmetric version from Theorem~\ref{thm:doubling4}.
\end{remark}

\subsection{Deduction of Proposition~\ref{prop:kneser-dense} from Proposition~\ref{prop:doubling4}}
\label{ssec:Kneser-densefromdoubling4}

We may assume that $N$ is large enough in terms of $\ee$, since otherwise the statement is obvious. Fix a positive integer $n_0 \in [N/2, N]$, and let $N' = \lfloor n_0/2 \rfloor$. Let $A_1, A_2 \subset [N']$ be the essential supports of $f_1,f_2$, respectively:
\[ A_i = \{n \in [N']: f_i(n) \geq \ee/2\}. \]
Let $\eta = \eta(\ee)$ be the constant from Proposition~\ref{prop:doubling4}. For each arithmetic progression $P \subset [N']$ of length $\geq \eta N$, since $f_i$ has average at least $1/3+\ee$ on $P$, it follows that
\[ |A_i \cap P| \geq (1/3 + \ee/2) |P|. \]
By Proposition~\ref{prop:doubling4} applied to $A_1,A_2$ with $\alpha = 1/3 + \ee/2$, we conclude that
\[ |S_{\eta}(A_1,A_2)| \geq (4/3+\ee) N' \geq 2n_0/3. \]
Thus $S_{\eta}(A_1,A_2)$ is a subset of $[n_0]$ with density at least $2/3$. On the other hand, the set
\[ A_3 = \{n \in [n_0]: f_3(n) \geq \ee/2\} \]
has density at least $1/3+\ee/2$ by assumption. To lower bound $f_1*f_2*f_3(n_0)$, note that the number of ways to write $n_0 = a+b$ with $a \in A_3$ and $b \in S_{\eta}(A_1,A_2)$ is at least
\[ |A_3| - (n_0 - |S_{\eta}(A_1,A_2)|) \geq \ee n_0/2 \gg \ee N. \]
Since $f_1*f_2(b) \gg \ee^2\eta \max(|A_1|, |A_2|) \gg \ee^2\eta N$ whenever $b \in S_{\eta}(A_1,A_2)$ and $f_3(a) \gg \ee$ whenever $a \in A_3$, it follows that
\[ f_1*f_2*f_3(n_0) \gg \ee^4 \eta N^2, \]
as desired. This finishes the proof of Proposition~\ref{prop:kneser-dense}.

\subsection{Deduction of Proposition~\ref{prop:kneser-sparse} from Proposition~\ref{prop:kneser-dense}}
\label{ssec:Kneser-sparsefromKneser-dense}

It remains to deduce Proposition~\ref{prop:kneser-sparse} from the special case $\nu = 1$. While the main idea is standard, it is a bit trickier to work in $\Z$ instead of a finite group. We follow the arguments from~\cite{Shao}.  The heart of the matter is to decompose $f$ into a uniform part, which contributes little to the three-fold convolution, and a structured part, which is bounded by $1$ pointwise. 


\begin{lemma}\label{lem:transference}
Let $f \in \{f_1,f_2,f_3\}$ be as in the statement of Proposition~\ref{prop:kneser-sparse}. Let $\delta \in (0,1)$ be a parameter. There is a decomposition $f = g+h$ satisfying the following properties:
\begin{enumerate}
\item the function $g$ satisfies the pointwise bound $0 \leq g(n) \leq 1 + O_{\delta}(\eta)$;
\item the function $h$ is Fourier uniform in the sense that $\|\widehat{h}\|_{\infty} \leq \delta N$;
\item for each arithmetic progression $P \subset [N]$ with $|P| \geq (\eta+\delta/\ee) N$ we have $\E_{n \in P} g(n) \geq 1/3 + \ee/2$;
\item we have $\|\widehat{g}\|_q \leq K N^{1-1/q}$ and $\|\widehat{h}\|_q \leq K N^{1-1/q}$.
\end{enumerate}
\end{lemma}


Assuming the existence of such a decomposition, we may conclude the proof of Proposition~\ref{prop:kneser-sparse} as follows. Fix a positive integer $n_0 \in [N/2,N]$. For each $i \in \{1,2,3\}$, let $f_i = g_i + h_i$ be the decomposition of $f_i$ from Lemma~\ref{lem:transference}, corresponding to a parameter $\delta > 0$ sufficiently small in terms of $\ee, K, q$. We may assume that $\eta$ is small enough in terms of $\ee, \delta$, since otherwise the conclusion can be made trivial. In particular, by property (1) we have the pointwise bound $g_i(n) \leq 1 + \ee/4$.

Decompose $f_1*f_2*f_3(n_0)$ into the sum of eight terms, one of them being the main term $g_1*g_2*g_3(n_0)$. By properties (1) and (3), we may apply Proposition~\ref{prop:kneser-dense} (after renormalizing so that each $g_i$ is $1$-bounded) to conclude that
\[ g_1*g_2*g_3(n_0) \geq c(\ee) N^2, \]
for some constant $c(\ee) > 0$ depending only on $\ee$. It remains to show that all the other seven terms are small. We only consider $h_1*h_2*h_3(n_0)$, as the other terms are treated similarly. We have
\[ |h_1*h_2*h_3(n_0)| \leq \int_0^1 |\widehat{h_1}(\gamma) \widehat{h_2}(\gamma) \widehat{h_3}(\gamma)| \dd\gamma \leq \|\widehat{h_1}\|_{\infty}^{3-q} \|\widehat{h_1}\|_q^{q-2} \|\widehat{h_2}\|_q \|\widehat{h_3}\|_q \leq \delta^{3-q} K^q N^2 \]
by H\"{o}lder's inequality and properties (2) and (4). This can be made smaller than $c(\ee)N^2/10$ by choosing $\delta$ small enough in terms of $\ee, K, q$. Thus the total contributions of the seven error terms are small compared to the main term and we have
\[ f_1*f_2*f_3(n_0) \geq \frac{1}{5} c(\ee) N^2 \]
as desired. It remains only to prove Lemma~\ref{lem:transference}.


\begin{proof}[Proof of Lemma~\ref{lem:transference}]
Let $T$ be the set of large frequencies of $f$:
\[ T = \{\gamma \in \R/\Z: |\widehat{f}(\gamma)| \geq \delta N\}, \]
and define a Bohr set using these frequencies:
\[ B=\{1\leq b\leq \delta N: \|b\gamma\|_{\R/\Z}<\delta/30 \text{ for all }\gamma\in T\}. \]
By the restriction estimate $\|\widehat{f}\|_q \leq K N^{1-1/q}$ and a standard covering argument, one easily deduces that $|B| \gg_{\delta} N$ (see~\cite[Lemma 3.2]{Shao}). Define the functions $g,h$ by
\[ g(n)=\E_{b_1,b_2\in B}f(n+b_1-b_2)\quad \text{and} \quad h(n)=f(n)-g(n). \]
We show that these functions satisfy the required properties. To prove the pointwise upper bound for $g$, note that
\begin{align*}
g(n) &\leq \E_{b_1,b_2\in B}\nu(n+b_1-b_2) = \E_{b_1,b_2\in B} \int_0^1 \widehat{\nu}(\gamma) \e((n+b_1-b_2)\gamma) \dd\gamma  \\
&= \int_0^1 \widehat{\nu}(\gamma) \e(n\gamma) \left|\E_{b\in B} \e(b\gamma)\right|^2 \dd\gamma.
\end{align*}
By the assumption on $\widehat{\nu}$, we may replace $\widehat{\nu}$ above by $\widehat{\1_{[N]}}$ at the cost of an error bounded by 
\[ \int_0^1 \left|\widehat{\nu}(\gamma) - \widehat{\1_{[N]}}(\gamma)\right| \left|\E_{b\in B} \e(b\gamma)\right|^2 \dd\gamma  \leq \eta N \int_0^1 \left|\E_{b\in B} \e(b\gamma)\right|^2 \dd\gamma \leq \frac{\eta N}{|B|} \ll_{\delta} \eta. \]
It follows that
\[ g(n) \leq \int_0^1 \widehat{\1_{[N]}}(\gamma) \e(n\gamma) \left|\E_{b\in B} \e(b\gamma)\right|^2 \dd\gamma + O_{\delta}(\eta) = \E_{b_1,b_2 \in B} \1_{[N]}(n+b_1-b_2) + O_{\delta}(\eta) \leq 1 + O_{\delta}(\eta), \]
where the middle equality follows by reversing the above process. This proves (1). 

To prove the Fourier uniformity of $h$,  note that
\begin{equation}\label{eq:transfer-h-hat} 
\widehat{h}(\gamma)=\widehat{f}(\gamma)\left(1-|\E_{b\in B}\e(b\gamma)|^2\right).
\end{equation}
If $\gamma\notin T$, then $|\widehat{h}(\gamma)|\leq
|\widehat{f}(\gamma)|\leq \delta N$ by the definition of $T$. If $\gamma\in T$, then
\[ 1-|\E_{b\in B}\e(b\gamma)|^2\leq 2(1-|\E_{b\in B}\e(b\gamma)|)\leq 2\E_{b\in B}|1-\e(b\gamma)|\leq \delta/2 \]
by the definition of $B$. Thus in either case we have $|\widehat{h}(\gamma)|\leq \delta N$, proving (2).

Property (3) inherits from the similar property satisfied by $f$, since $B \subset [1,\delta N]$. More precisely, for any arithmetic progression $P \subset [N]$ with $|P| \geq (\eta + \delta/\ee) N$ we have
\[ \sum_{n \in P} g(n) = \E_{b_1,b_2 \in B} \sum_{n \in P} f(n+b_1-b_2) = \E_{b_1,b_2 \in B} \sum_{n \in P+b_1-b_2} f(n). \]
Since $P+b_1-b_2$ is again an arithmetic progression, whose intersection with $[N]$ has size at least $|P| - \delta N \geq \eta N$, each inner sum over $n$ above is at least
\[ (1/3+\ee)(|P| - \delta N) = |P| \left(1/3+\ee - \frac{\delta N}{2|P|}\right) \geq |P|(1/3 + \ee/2). \] 
This proves property (3). The final property (4) follows immediately from~\eqref{eq:transfer-h-hat}.
\end{proof}


\section{The sieve majorant}\label{sec:rho+}


In this section we establish the condition $(1)$ of Theorem~\ref{thm:vin3} with $\theta = 11/20 + 2\varepsilon$ by constructing the required sieve majorant $\rho^+$. The argument is continuous and with slight modifications would work also for $\theta = 11/20 - \varepsilon$. Note that if one takes $\rho^+$ to be a standard upper bound sieve (such as Selberg's sieve), then $\alpha^+$ is at least $2/\theta$ since the short interval has exponent of distribution $\theta$, which is not enough for our purposes. One may attempt to rectify this issue by using a well-factorable error term from the linear sieve so that the short interval has a much larger level of bilinear distribution, but we are unable to control this error term when twisted by an exponential phase.

Instead, following the arguments in~\cite{Baker-Harman-4/7}, we define $\rho^+$ based on Harman's sieve. Let $\rho(n,y) = 1$ if $(n,P(y)) = 1$ and $\rho(n,y) = 0$ otherwise, and take
\begin{equation}\label{eq:rho+1} 
\rho^+(n) = \rho(n, x^{1/4}) + \sum_{\substack{n = p_1p_2p_3m \\ z< p_1 < p_2 < p_3 < x^{1/4}}} \rho(m, p_1),\qquad z=x^{1/10}.
\end{equation} 
This is a simpler version of the majorant constructed in~\cite{Baker-Harman-4/7}, but it is already enough for our purposes (and for theirs as well). The Fourier transform of $\rho^+$ can be understood via Proposition~\ref{prop:SAz} below.

In what follows $\ee > 0$ and $A \geq 2$ are always fixed, assumed to be sufficiently small and sufficiently large, respectively. In particular when we have conditions like $d \leq (\log x)^{O(1)}$ or $|a_m| \leq \tau(m)^{O(1)}$, we assume that $A$ is large compared to the implied constants. We will also repeatedly use $B$ for an arbitrarily large constant. Moreover, $x$ is assumed to be sufficiently large in terms of $\ee$, $A$ and $B$, and $o(1)$ denotes a quantity that tends to $0$ as $x\to\infty$. To facilitate proofs, we will use smoothed sums, and compare sums over the short interval $[x,x+x^{\theta}]$ to their counterparts over the long interval $[x,x+h_1]$, where
\[ h_1 = x \exp(-\log^{1/2}x). \]


\begin{proposition}
\label{prop:SAz}
Let $\theta > 11/20$, let $g$ be a smooth function supported on $[x,x+x^\theta]$ satisfying $|g^{(j)}(t)| \ll_j x^{-j\theta}$ for all $j\ge 0$, and let $g_1$ be a smooth function supported on $[x,x+h_1]$ with $h_1=x\exp(-\log^{1/2}{x})$ satisfying $|g_1^{(j)}(t)|\ll_j h_1^{-j}$ for all $j \geq 0$. Let $\gamma=a/q+\lambda$ for some $(a,q)=1$ with $q\le Q=x^{2\theta-1}\log^{-5A}{x}$ and $|\lambda|<1/q Q$ and let $(c,d)=1$ with $d\ll \log^{O(1)}{x}$. Then we have
\[
\sum_{\substack{n\in [x,x+x^\theta] \\ n\equiv c\pmod{d} }}\rho^+(n)g(n)\e(n\gamma)=\frac{[d,q]}{\varphi([d,q])}\Bigl(\frac{1}{\widehat{g_1}(0)}\sum_{n\in [x,x+h_1]}g_1(n)\rho^+(n)\Bigr)\Bigl(\sum_{\substack{n\in [x,x+x^\theta]\\ n\equiv c\pmod{d} \\ (n,q)=1}}g(n)\e(\gamma n)\Bigr)+O\Bigl(\frac{x^\theta}{\log^{A}{x}}\Bigr).
\]
Moreover, if $q\ge \log^{5A}{x}$ then both sides of this equation are $O(x^\theta/\log^A{x})$.
\end{proposition}


It is now simple to establish condition $(1)$ of Theorem \ref{thm:vin3} from Proposition \ref{prop:SAz}. To simplify the notation we establish this condition for the interval $[x, x+x^\theta]$ but doing it for $[x-x^\theta/3, x+x^\theta]$ would make no difference. 

By the expression~\eqref{eq:rho+1} and the prime number theorem (compare with~\cite[Section 1.4]{Harman-book}),
\[
\frac{1}{\widehat{g}_1(0)} \sum_{x \leq n \leq x + h_1} g_1(n) \rho^+(n) = \frac{1+o(1)}{\log x}\left(4\omega(4) + \int_{1/10}^{1/4} \int_{\beta_1}^{1/4} \int_{\beta_2}^{1/4} \frac{\omega(\frac{1-\beta_1-\beta_2-\beta_3}{\beta_1})}{\beta_1^2\beta_2\beta_3} \dd\beta_1 \dd\beta_2 \dd\beta_3\right),
\]
where $\omega$ is the Buchstab function. Numerical calculation\footnote{A Mathematica\textregistered\, file performing this computation is included along with this document on \url{www.arxiv.org}.} using the fact that $\omega(u) \leq \omega(3)$ for all $u \geq 3$ (see~\cite[Equation (1.4.16)]{Harman-book}) shows that, writing $\alpha^+$ for the expression in parentheses, we have $\alpha^+ < 2.9$. 

Let $\eta > 0$ be fixed and small enough. Let $d=W=\prod_{p\le 0.1\log\log{x}}p$ and let $g$ be a smooth function supported on $[x,x+x^\theta]$ which is 1 on $[x+\eta x^\theta,x+(1-\eta )x^\theta]$ and satisfies $g^{(j)}(t)\ll_{j} x^{-j\theta}$. For any $\gamma \in \R$, we may write $\gamma = a/q + \lambda$ for some $(a,q)=1$ with $q \leq Q$ and $|\lambda| < 1/qQ$ as in the statement of Proposition~\ref{prop:SAz}. Then, by Proposition \ref{prop:SAz} we get
\[
\sum_{\substack{ n \in[x, x + x^\theta]  \\ n\equiv c\pmod{W} }} g(n) \rho^+(n) e\left(\gamma n \right) = \frac{\alpha^+ + o(1)}{\log x} \cdot \frac{[W,q]}{\varphi([W,q])} \sum_{\substack{n \in[x, x + x^\theta] \\ n\equiv c\pmod{W} \\  (n, q) = 1}} g(n) e\left(\gamma n \right) + O\left(\frac{x^{\theta}}{\log^Ax}\right),
\]
and if $q\ge \log^{5A}{x}$ then all terms in this equation are $O(x^\theta/\log^A{x})$. By a simple sieve upper bound, we see that the contribution to the left hand side from $n\in [x,x+\eta x^\theta]$ and $n\in [x+(1-\eta)x^\theta,x+x^\theta]$ is $O(\eta x^\theta/\varphi(W)\log{x})$. Thus, we find
\begin{align*}
\sum_{\substack{ n \in[x, x + x^\theta]  \\ n\equiv c\pmod{W} }}\rho^+(n) e\left(\gamma n \right) &= \sum_{\substack{ n \in[x, x + x^\theta]  \\ n\equiv c\pmod{W} }} g(n)\rho^+(n) e\left(\gamma n \right)+O\Bigl(\frac{\eta x^\theta}{\varphi(W) \log{x}}\Bigr)\\
&=\frac{\alpha^+ + o(1)}{\log x} \cdot \frac{[W,q]}{\varphi([W,q])} \sum_{\substack{n \in[x, x + x^\theta] \\ n\equiv c\pmod{W} \\  (n, q) = 1}} g(n) e\left(\gamma n \right)+O\Bigl(\frac{\eta x^\theta}{\varphi(W) \log{x}}\Bigr)\\
&=\frac{\alpha^+ + o(1)}{\log x} \cdot \frac{[W,q]}{\varphi([W,q])} \sum_{\substack{n \in[x, x + x^\theta] \\ n\equiv c\pmod{W} \\  (n, q) = 1}}e\left(\gamma n \right)+O\Bigl(\frac{\eta x^\theta}{\varphi(W) \log{x}}\Bigr).
\end{align*}
If $q\ge \log^{5A}{x}$ then all terms are $O(\eta x^\theta/\varphi(W)\log{x})$, which gives Theorem \ref{thm:vin3} in this case. If instead $q\le \log^{5A}{x}$ then $[W,q]/\varphi([W,q])=(1+o(1))W/\varphi(W)$ and the contribution to the summation from $n\equiv c\pmod{W}$ with $(n,q)\ne 1$ is $o(x^\theta/\varphi(W)\log{x})$. From this condition $(1)$ of Theorem \ref{thm:vin3} follows immediately.


It remains to prove Proposition~\ref{prop:SAz}. The proof is very similar to the proof in~\cite[Section 3]{Baker-Harman-4/7} but for completeness we give a relatively self-contained proof. We split the proof of Proposition~\ref{prop:SAz} into two parts according to the size of $q$ appearing in the Diophantine approximation for $\gamma$. We let $Q=x^{2\theta-1}\log^{-5A}{x}$ and put
\[
\mathcal{M}=\bigcup_{q\le \log^{5A}{x}}\bigcup_{(a,q)=1}\Bigl[\frac{a}{q}-\frac{1}{q Q},\frac{a}{q}+\frac{1}{q Q}\Bigr].
\]
We first decompose $\rho^+$ into simpler Type I and Type II sequences in Section~\ref{sec:SieveDecomposition}, and then establish Proposition~\ref{prop:SAz} for $\gamma\notin\mathcal{M}$ in Section~\ref{sec:Minor} and for $\gamma\in\mathcal{M}$ in Section~\ref{sec:Major}.

\subsection{Sieve decomposition}\label{sec:SieveDecomposition}

In this section we give a combinatorial decomposition (Lemma~\ref{lem:SieveDecomposition} below) of the function $\rho^+$, showing that it can be rewritten as a sum of `Type I' terms (sequences of the form $a*1$ where $a$ is supported on small numbers) and special `Type II' terms (sequences approximately of the form $a*b*\rho$ where $a$ and $b$ have support in dyadic ranges close to $x^{1/2}$, and $\rho$ is the indicator function of the primes.) In what follows we always set
\[ \omega=\exp((\log{x})^{9/10}). \]


\begin{lemma}[Fundamental Lemma]\label{lem:fundamentallem}
Let $\{a_m\}$ be a complex sequence supported on $m\le x^{\theta-2\varepsilon}$ satisfying $|a_m|\ll \tau(m)^{O(1)}$ .  Let
\[ \ee(n) = \sum_{mh=n} a_m \rho(h, \omega) - \sum_{\substack{mdh = n \\ d < x^{\ee}, d|P(\omega)}} a_m\mu(d). \] 
Then we have
\[ 
\sum_{n \in [x,x+x^{\theta}]} |\ee(n)| \ll \frac{x^{\theta}}{\log^{B}x}
\]
for any $B \geq 1$.
\end{lemma}


\begin{proof}
By \cite[Lemma 15]{HeathBrown} we have
\[
\rho(h,\omega)=\sum_{d|h,\, d|P(\omega)}\mu(d)=\sum_{\substack{d|h,\, d|P(\omega)\\ d<x^\varepsilon}}\mu(d)+O\Bigl(\sum_{\substack{d|h,\,d|P(\omega)\\ x^\varepsilon<d<\omega x^\varepsilon}}1\Bigr).
\]
It follows that
\[ \ee(n) \ll \sum_{\substack{mdh = n \\ d|P(\omega),x^{\ee}<d<\omega x^{\ee}}} |a_m| \ll \sum_{\substack{mdh = n \\ d|P(\omega),x^{\ee}<d<\omega x^{\ee} \\ m\leq x^{\theta-2\ee}}} (\tau(m))^{O(1)}. \] 
Performing the summation over $n$, we obtain
\[
\sum_{n \in [x,x+x^{\theta}]} |\ee(n)| \ll \sum_{\substack{m dh\in[x,x+x^\theta] \\ d|P(\omega),\, x^\varepsilon<d<\omega x^\varepsilon \\ m\le x^{\theta-2\varepsilon}}}\tau(m)^{O(1)}\ll x^\theta\Bigl(\sum_{m\le x^{\theta-2\varepsilon}}\frac{\tau(m)^{O(1)}}{m}\Bigr)\Bigl(\sum_{\substack{d|P(\omega) \\ x^\varepsilon<d<\omega x^\varepsilon}}\frac{1}{d}\Bigr).
\]
The first sum in the parentheses is $O((\log{x})^{O(1)})$. It follows from the density of smooth numbers (see \cite[Theorem 1]{Hildebrand}, for example) that the second sum is
\[
\ll \exp\Bigl(-\frac{1}{8}u\varepsilon\log (u\varepsilon)\Bigr)\ll \exp(-(\log{x})^{1/10}),
\]
where $u=\log{x}/\log{\omega}=(\log{x})^{1/10}$. This gives the result.
\end{proof}

\begin{lemma}[Sieve decomposition]\label{lem:SieveDecomposition}
There exists complex sequences $c^{(I)}_j(n)$, $c^{(II)}_{j}(n)$ for $1\le j\le J\ll \log^{O(1)}{x}$ such that the following holds.
\begin{itemize}
\item Approximation to $\rho^+$: Let $\theta \in [1/2+2\ee, 1]$. We have
\[ \sum_{n \in [x,x+x^{\theta}]} \left| \rho^+(n) - \sum_{j\leq J}(c_j^{(I)}(n) + c_j^{(II)}(n)) \right| \ll \frac{x^{\theta}}{\log^B x} \]
for any $B \geq 1$.
\item $c_j^{(I)}$ are Type I coefficients: There exists a sequence $b_j(m)\ll \tau(m)^{O(1)}$ supported on $(m,P(\omega))=1$ such that
\[
c_j^{(I)}(n)=\sum_{\substack{m|n \\ m \leq x^{1/2+\varepsilon} }}b_{j}(m).
\]
\item $c_j^{(II)}$ are Type II coefficients: There exists sequences $e_j(m;\mathbf{t}),f_j(m;\mathbf{t})\ll \tau(m)^{O(1)}$ supported on $(m,P(\omega))=1$ which are smooth functions of $\mathbf{t}=(t_1,t_2,t_3)$ such that
\[
c_j^{(II)}(n)=\int_{-x^{1+\varepsilon}}^{x^{1+\varepsilon}} \int_{-x^{1+\varepsilon}}^{x^{1+\varepsilon}} \int_{-x^{1+\varepsilon}}^{x^{1+\varepsilon}}\sum_{\substack{ m l p=n \\ m\sim E\\ l\sim F \\ p\sim P}}e_{j}(m;\mathbf{t})f_j(l;\mathbf{t})p^{i(t_1+t_2+t_3)} K(t_1, t_2 ,t_3)dt_1 dt_2 dt_3
\]
for some function $K(t_1,t_2,t_3)\ll (\log{x})^{O(1)}/\prod_{i=1}^{3}(1+|t_i|)$ and quantities $E=x^{\beta_1}$, $F=x^{\beta_2}$, $P\asymp x^{1-\beta_1-\beta_2}>\exp(\log^{9/10}{x})$ satisfying $|\beta_1-\beta_2|<1/10+\varepsilon$ and $\beta_1+\beta_2>9/10-\varepsilon$.
\end{itemize}
\end{lemma}


\begin{proof}
Our argument is essentially identical to \cite[Lemma 3.8]{BHP97} with slightly different parameters, but we give a proof for completeness. From the definition of $\rho^+$ in~\eqref{eq:rho+1} and Buchstab's identity, we obtain
\begin{align*}
\rho^+(n)&=\rho(n,z)-\sum_{\substack{ph=n\\ z<p<x^{1/4}}}\rho(h,z)+\sum_{\substack{p_1p_2h=n\\ z<p_2<p_1<x^{1/4}}}\rho(h,z)\\
&=\sum_{m h=n}a_m\rho(h,z)
\end{align*}
for some coefficients $|a_m|\ll \tau(m)$ supported on $m\le x^{1/2}$. Let
\begin{align*}
u_j(n)=\sum_{\substack{mp_1\dots p_j h=n \\ \omega< p_j<\dots<p_1<z \\ mp_1\dots p_{j-1}p_j^{1/2}<x^{9/20} }}a_m\rho(h,p_j),\\
v_j(n)=\sum_{\substack{mp_1\dots p_j h=n \\ \omega< p_j<\dots<p_1<z \\ mp_1\dots p_{j-1}p_j^{1/2}<x^{9/20} }}a_m\rho(h,\omega),\\
w_j(n)=\sum_{\substack{mp_1\dots p_j h=n \\ \omega< p_j<\dots<p_1<z \\ mp_1\dots p_{j-2}p_{j-1}^{1/2}<x^{9/20} \\ mp_1\dots p_{j-1}p_j^{1/2}\ge x^{9/20} }}a_m\rho(h,p_j).
\end{align*}
Then Buchstab's identity gives
\[
u_j(n)=v_j(n)-w_{j+1}(n)-u_{j+1}(n).
\]
Clearly $u_j(n)=v_j(n)=w_j(n)=0$ for all $n\ll x$ if $j\ge 2\log{x}$, since $p_1\dots p_j\ge x^{2+o(1)}$ in this case. Thus, by first applying Buchstab's identity once and then repeatedly applying the above identity, we obtain
\begin{align*}
\rho^+(n)=\sum_{m h=n}a_m\rho(h,z)&=\sum_{m h=n}a_m\rho(h,\omega)-u_1(n)-w_1(n)\\
&=\sum_{m h=n}a_m\rho(h,\omega)+\sum_{j=1}^{\lfloor 2\log{x}\rfloor}(-1)^j(v_j(n)+w_j(n)).
\end{align*}
By Lemma~\ref{lem:fundamentallem} we have
\[
v_j(n)=\sum_{\substack{mp_1\dots p_j h=n \\ \omega < p_j<\dots<p_1<z \\ mp_1\dots p_{j-1}p_j^{1/2}<x^{9/20} }}a_m\rho(h,\omega)=\sum_{\substack{mp_1\dots p_j d h=n \\ \omega < p_j<\dots<p_1<z \\ mp_1\dots p_{j-1}p_j^{1/2}<x^{9/20}\\ d<x^{\varepsilon},\, d|P(\omega) }}a_m\mu(d)+\varepsilon_j(n)
\]
for some $\varepsilon_j(n)$ satisfying $\sum_{n\in[x,x+x^\theta]}|\varepsilon_j(n)|\ll x^\theta \log^{-B}{x}$ for any $B \geq 1$. We then see that $c_j^{(I)}(n):=v_j(n)-\varepsilon_j(n)$ satisfies the criteria of the Lemma, by taking
\[
b_j(k)=\sum_{\substack{mp_1\dots p_j d =k \\ \omega < p_j<\dots<p_1<z \\ mp_1\dots p_{j-1}p_j^{1/2}<x^{9/20}\\ d<x^{\varepsilon},\, d|P(\omega) }}a_m\mu(d),
\]
and noting that $mp_1\dots p_j d <x^{9/20}z^{1/2}x^{\varepsilon}<x^{1/2+\varepsilon}$. Similarly the term $\sum_{m h=n}a_m\rho(h,\omega)$ is also a Type I coefficient plus a negligible error.

Thus, to complete the proof we need to show each $w_j(n)$ can be written as a combination of Type II coefficients and a small error. We first decompose $w_j(n)$ into $O(\log^3{x})$ terms $w_{j;\beta_1,\beta_2,\beta_3}(n)$ for different $(\beta_1,\beta_2,\beta_3)\in\mathbb{R}^3_{>0}$ where we have the additional dyadic restrictions $h\sim x^{\beta_1}$, $mp_1\dots p_{j-1}\sim x^{\beta_2}$ and $p_j\sim x^{\beta_3}$ in the summation of $w_j(n)$. Let us first see that $w_{j;\beta_1,\beta_2,\beta_3}(n)=0$ for all $n\asymp x$ unless $(\beta_1,\beta_2,\beta_3)$ satisfy the inequalities of the Lemma. In particular, since $m p_1\dots p_j h=n\asymp x$ we have that $w_{j;\beta_1,\beta_2,\beta_3}(n)=0$ unless $\beta_1+\beta_2+\beta_3=1+o(1)$. Since $p_j<z=x^{1/10}$, we have that $w_{j;\beta_1,\beta_2,\beta_3}(n)=0$ unless $1-\beta_1-\beta_2 = \beta_3 + o(1)<1/10+\varepsilon$. Similarly, since 
\[
x^{\beta_2-\beta_1}\asymp \frac{mp_1\dots p_{j-1}}{h}\asymp \frac{(mp_1\dots p_{j-1}p_j^{1/2})^2}{x}\gg x^{-1/10}
\]
by the restriction $mp_1\dots p_{j-1}p_j^{1/2}>x^{9/20}$ in $w_j(n)$, and
\[
x^{\beta_2-\beta_1} \asymp \frac{(mp_1\dots p_{j-1}p_j^{1/2})^2}{x}\ll x^{-1/10} p_{j-1} p_j \ll x^{-1/10}z^2\ll x^{1/10}
\]
by the restrictions $mp_1\dots p_{j-2}p_{j-1}^{1/2}<x^{9/20}$ and $p_j < p_{j-1}<z$, we see that $w_{j;\beta_1,\beta_2,\beta_3}(n)=0$ unless $|\beta_2 - \beta_1| < 1/10+\varepsilon$. This verifies that we only need consider $\beta_1+\beta_2>9/10-\varepsilon$ and $|\beta_1-\beta_2|<1/10+\varepsilon$, as required. 

We now decompose $w_{j;\beta_1,\beta_2,\beta_3}(n)$ further to remove the dependencies between $p_j$ or $h$ and the remaining variables. Let $w_{j;\beta_1,\beta_2,\beta_3;r}(n)$ denote the sum $w_{j;\beta_1,\beta_2,\beta_3}(n)$ with the additional restriction that $h=p'_1\dots p'_r$ for some primes $p_i'$ with $p_1'\ge \dots \ge p'_r$. Then we may replace $\rho(h,p_j)$ with 1 at the cost of the extra condition that $p_r'>p_j$. Clearly $w_{j;\beta_1,\beta_2,\beta_3;r}(n)=0$ if $r\ge \log{x}$ since $h<x$, so this decomposes $w_{j;\beta_1,\beta_2,\beta_3}(n)$ into $O(\log{x})$ further terms.

Finally, we apply Perron's formula three times to remove the multiplicative dependencies between $p_1'\dots p_r'$, $mp_1\dots p_{j-1}$ and $p_j$ which occur in the inequalities $p_j<p_r'$, $p_j<p_{j-1}$ and $mp_1\dots p_{j-1}p_j^{1/2} \geq x^{9/20}$. For example, since $p_j,p_{j-1}\ll x$, Perron's formula gives
\[
\sum_{\substack{m,p_1,\dots,p_j,p_1'\dots  p_r'\\ p_j<p_{j-1}}}^*a_m=\frac{1}{2\pi i}\int_{-T}^{T}\sum^*_{m,p_1,\dots,p_j,p_1'\dots, p_r'}a_m\Bigl(\frac{p_{j-1}-1/2}{p_j}\Bigr)^{c+it}\frac{dt}{c+it}+O\Bigl(\frac{x^{1+o(1)}}{T}\Bigr),
\]
where $c\asymp 1/\log{x}$ and by $\sum^*$ we indicate that we have suppressed the further conditions of summation. (We have used $p_{j-1}-1/2$ to ensure the variables are never within more than 1/2 of each other.) We see that the error term contributes $O(x^{1+\theta+o(1)}/T)$ to the sum over $n\in[x,x+x^\theta]$ (using the fact that $p_j,p_{j-1}\ll x$), and so can be ignored if we take $T=x^{1+\varepsilon}$. Repeatedly applying this to remove the three inequalities listed above, and choosing the three values $c_i \asymp 1/\log{x}$ suitably to cancel the real part of the exponent of $p_j$, we are left with a suitable Type II coefficient, and so have completed the decomposition.
\end{proof}

\subsection{Minor Arcs}\label{sec:Minor}

We now establish Proposition~\ref{prop:SAz} when $\gamma$ does not have a rational approximation with small denominator. In this case the frequencies $\e(\gamma n)$ equidistribute so the result follows quickly for any Type I or Type II sequence from standard bounds on exponential sums.


\begin{lemma}[Type I sums]\label{lem:TypeIMinor}
Let $a_m\ll \tau(m)^{O(1)}$ be a complex sequence supported on $m\le x^{\theta-\varepsilon}$, and let $g$ be a smooth function supported on $[x,x+x^\theta]$ with $|g(t)|\ll 1$ and $|g'(t)|\ll x^{-\theta}$. Let $\gamma \in \mathcal{M}^c$ and $d\ll \log^{O(1)}{x}$. Then
\[
\sum_{\substack{m n\in[x,x+x^\theta]\\ m n\equiv c\pmod{d}}}a_mg(m n)\e(m n\gamma)\ll \frac{x^\theta}{\log^{2A}{x}}.
\]
\end{lemma}


\begin{proof}
This follows from well-known exponential sum estimates. Since $\gamma\in\mathcal{M}^c$, there is a $q \in [\log^{5A}{x}, Q]$ such that $\|\gamma q\|<1/q^2$. Thus
\begin{align*}
\sum_{\substack{m n\in[x,x+x^\theta]\\ mn\equiv c\pmod{d}}}a_m g(mn) \e(m n\gamma)&\ll \sum_{m\leq x^{\theta-\varepsilon}}\tau(m)^{O(1)}\Bigl|\sum_{\substack{n\in[x/m,x/m+x^\theta/m] \\ mn\equiv c\pmod{d} }}g(mn) \e(m n\gamma)\Bigr|\\
&\ll \Bigl(\sum_{m\leq x^{\theta-\varepsilon}}\min\Bigl(\frac{x^\theta}{m},\|dm\gamma\|^{-1}\Bigr)\Bigr)^{1/2}\Bigl(x^\theta \sum_{m\leq x^{\theta-\varepsilon}}\frac{\tau(m)^{O(1)}}{m}\Bigr)^{1/2}\\
&\ll x^{\theta/2}\Bigl(\frac{x^\theta}{q}+x^{\theta-\varepsilon}+q\Bigr)^{1/2}(\log{x})^{O(1)}\ll \frac{x^\theta}{\log^{2A}{x}}.\qedhere
\end{align*}
\end{proof}


\begin{lemma}[Type II sums]\label{lem:TypeIIMinor}
Let $a_m,b_m\ll \tau(m)^{O(1)}$ and let $g$ be a smooth function supported on $[x,x+x^\theta]$ with $|g(t)|\ll 1$ and $|g'(t)|\ll x^{-\theta}$. Let $\gamma\in\mathcal{M}^c$ and $d\ll \log^{O(1)}{x}$. If $x^{1-\theta+\varepsilon}\ll M\ll x^{\theta-\varepsilon}$ then we have
\[
\sum_{\substack{m n\in[x,x+x^\theta]\\ m n\equiv c\pmod{d} \\ m\sim M}}a_m b_n g(mn) \e(m n\gamma)\ll \frac{x^\theta}{\log^{2A}{x}}.
\]
\end{lemma}


\begin{proof}
The proof is again standard. By applying the Cauchy-Schwarz inequality, we have
\begin{align*}
&\Bigl|\sum_{\substack{m n\in[x,x+x^\theta]\\ m n\equiv c\pmod{d} \\ m\sim M}}a_m b_n g(m n)\e(m n\gamma)\Bigr|^2\ll M(\log{x})^{O(1)}\sum_{m\sim M}\Bigl|\sum_{\substack{n\in [x/m,x/m+x^\theta/m]\\ m n\equiv c\pmod{d} }}b_n g(m n)\e(\gamma m n)\Bigr|^2\\
&=M(\log{x})^{O(1)}\sum_{n_1,n_2\asymp x/M}\tau(n_1)^{O(1)}\tau(n_2)^{O(1)}\sum_{\substack{m\sim M\\ m n_1,m n_2\in[x,x+x^\theta] \\ m n_1\equiv m n_2\equiv c\pmod{d} }}g(m n_1)\overline{g(m n_2)}\e(m(n_1-n_2)\gamma)\\
&\ll M(\log{x})^{O(1)}\sum_{n_1\asymp x/M}\tau(n_1)^{O(1)}\sum_{\substack{n_2\asymp x/M\\ |n_2-n_1|\ll x^\theta/M}}\min\Bigl(\frac{ x^\theta}{x/M},\|d(n_1-n_2)\gamma\|^{-1}\Bigr)\\
&\ll M\frac{x}{M}(\log{x})^{O(1)}\Bigl(\frac{x^\theta}{x/M}+q\Bigr)\Bigl(1+\frac{x^\theta}{q M}\Bigr)= (\log{x})^{O(1)}\Bigl(M x^\theta+x q+\frac{x^{\theta+1}}{M}+\frac{x^{2\theta}}{q}\Bigr).
\end{align*}
Recalling that $q\ll x^{2\theta-1}\log^{-5A}{x}$ and $x^{1-\theta+\varepsilon}<M<x^{\theta-\varepsilon}$ we see this is $O(x^{2\theta}\log^{-4A}{x})$, giving the result.
\end{proof}


Combining Lemma~\ref{lem:TypeIMinor} and Lemma~\ref{lem:TypeIIMinor} with our sieve decomposition from Lemma~\ref{lem:SieveDecomposition} allows us to prove Proposition~\ref{prop:SAz} when $\gamma\in\mathcal{M}^c$, given by Lemma~\ref{lem:Minor} below.


\begin{lemma}[Proposition~\ref{prop:SAz} for $\gamma$ in minor arcs]\label{lem:Minor}
Let $g$ be a smooth function supported on $[x,x+x^\theta]$ with $|g^{(j)}(t)|\ll_j x^{-j\theta}$ for all $j$. Let $\gamma \in \mathcal{M}^c$ and $d\ll \log^{O(1)}{x}$. Then we have for $\theta\ge 11/20+2\varepsilon$ that
\[
\sum_{\substack{n\in [x,x+x^\theta] \\ n\equiv c\pmod{d} }}\rho^+(n)g(n)\e(n\gamma) \ll \frac{x^{\theta}}{\log^Ax}. 
\]
\end{lemma}


\begin{proof}
By Lemma~\ref{lem:SieveDecomposition} we obtain sequences $c_j^{(I)}(n)$ and $c_j^{(II)}(n)$ such that
\[
\sum_{\substack{n\in [x,x+x^{\theta}] \\ n\equiv c\pmod{d} }}\rho^+(n)g(n)\e(\gamma n)=\sum_{j\le \log^{O(1)}{x} }\sum_{\substack{n\in [x,x+x^\theta] \\ n\equiv c\pmod{d} }}(c_j^{(I)}(n)+c_j^{(II)}(n))g(n)\e(\gamma n)+O\Bigl(\frac{x^\theta}{\log^A{x}}\Bigr).
\]
We have $c_j^{(I)}(n)=\sum_{m h=n}b_j(m)$ for some $b_j(m)\ll \tau(m)^{O(1)}$ supported on $m\le x^{1/2+\varepsilon}$. Thus, by Lemma~\ref{lem:TypeIMinor}, we have
\[
\sum_{\substack{n\in[x,x+x^\theta]\\ n\equiv c\pmod{d} }}c_j^{(I)}(n)g(n)\e(\gamma n)\ll \frac{x^{\theta}}{\log^{2A}{x}}.
\]
We have that
\[
\sum_{\substack{ n\in[x,x+x^\theta] \\ n\equiv c\pmod{d} }}c_j^{(II)}(n)g(n)\e(\gamma n)\ll \log^{O(1)}{x}\sup_{d_p,e_m,f_l}\Bigl|\sum_{\substack{mlp\in[x,x+x^\theta]\\ mlp\equiv c\pmod{d} \\ m\sim E,\, l\sim F,\, p\sim P}} e_m f_l d_p g(m l p) \e(\gamma m l p)\Bigr|,
\]
where the supremum is over all sequences $e_m,f_m,d_m\ll \tau(m)^{O(1)}$. Moreover, $E F P\asymp x$, and $P,E/F,F/E\ll x^{1/10+\varepsilon}$, so $x^{9/20-\varepsilon/2}\ll P\min(F,E)\ll x^{11/20+\varepsilon/2}$. Thus, by Lemma~\ref{lem:TypeIIMinor}, we have
\[
\sum_{\substack{n\in[x,x+x^\theta]\\ n\equiv c\pmod{d} }}c_j^{(II)}(n)g(n)\e(\gamma n)\ll \frac{x^\theta}{\log^{3A/2}{x}}.
\]
Thus the claim follows.
\end{proof}

\subsection{Major arcs}\label{sec:Major}

We now complete the proof of Proposition~\ref{prop:SAz} by dealing with the case when $\gamma\in \mathcal{M}$, and so $\gamma$ does have a rational approximation with small denominator. In this case large and mean value estimates for Dirichlet polynomials can control the Mellin transform of the relevant sums for Type I coefficients or our shape of Type II coefficients. The Mellin transform of a function $h \colon \mathbb{R}^+ \to \mathbb{C}$ is defined to be
\[
H(s) = \int_0^\infty h(u) u^{s-1} du
\]
and its inverse is
\[
h(u) = \frac{1}{2 \pi i} \int_{c-i\infty}^{c+i\infty} H(s) u^{-s} ds.
\]
We begin with a preparatory lemma bounding the Mellin transform of $g(u) e(\lambda(u))$ when $g$ is a smooth function.


\begin{lemma}[Stationary phase]\label{lem:Stationary}
Let $g$ be a smooth function supported on $[x,x+x^\theta]$ satisfying $g^{(j)}(t)\ll_j x^{-j\theta}$ for all $j\ge 0$. Let $0<\lambda<x^{1-2\theta+o(1)}$ and
\[
G_\lambda(s)=\int_0^\infty g(u)u^{s-1}\e(\lambda u)du.
\]
Then for any $j\in\mathbb{Z}_{\ge 0}$ we have
\[
|G_\lambda(1/2+it)|\ll_j x^{\theta-1/2}\Bigl(\frac{x^{1-\theta+o(1)}}{1+|t+2\pi \lambda x|}\Bigr)^{j}.
\]
\end{lemma}


\begin{proof}
Let $h(t)=\log(1+t)-t=-\sum_{j\ge 2}(-t)^j/j$. We see that
\[
\lambda u=\lambda x+\lambda x \log{u}-\lambda x\log{x}-\lambda x h\Bigl(\frac{u-x}{x}\Bigr),
\]
so that
\[
G_\lambda(1/2+it)=\e(\lambda x)x^{-2\pi i \lambda x}\int_x^{x+x^\theta}g(u)u^{it+2\pi i \lambda x-1/2}\e\Bigl(-\lambda x h\Bigl(\frac{u-x}{x}\Bigr)\Bigr)du.
\]
To obtain our bound we repeatedly integrate by parts $j$ times. We note that for $|u-x|\ll x^\theta$ there holds
\[
\frac{\partial^j}{\partial u^j}h\Bigl(\frac{u-x}{x}\Bigr)\ll_j
\begin{cases}
x^{2\theta-2-j\theta},\qquad &j\in\{0,1\},\\
x^{-j},\qquad &j\ge 2,
\end{cases}
\]
so that, using $\lambda\ll x^{1-2\theta+o(1)}$ and $\theta<1$, we have
\[
 \frac{\partial^j}{\partial u^j}\lambda x h\Bigl(\frac{u-x}{x}\Bigr)\ll_j x^{-j\theta+o(1)}.
\]
Thus, since $g^{(j)}(t)\ll x^{-j\theta}$ as well, we find that
\[
\frac{\partial^j}{du^j}\Bigl(g(u)\e\Bigl(\lambda x h\Bigl(\frac{u-x}{x}\Bigr)\Bigr)\Bigr) \ll_j x^{-j\theta+o(1)}.
\]
Furthermore, for $y \ll x$, we have
\[
\idotsint\limits_{x<u_1<u_2<\dots <u_j<y}u_1^{-1/2+i(t+2\pi\lambda x)}du_1\dots du_j \ll_j \frac{x^{j-1/2}}{(1+|t+2\pi\lambda x|)^j}.
\]
Using these bounds in the integration by parts gives the result.
\end{proof}


\begin{lemma}[Type I Integrals]\label{lem:TypeIIntegral}
Let $a_m\ll \tau(m)^{O(1)}$ be a complex sequence supported on $m\sim M$. Let $h$ be a smooth function supported on $[x/4M,4x/M]$ which satisfies $|h^{(j)}(t)|\ll (M/x)^j$ for $j\in\{0,1,2\}$. Let $H(s)$ be the Mellin transform of $h$. Let $\chi$ be a Dirichlet character mod $q\ll \log^{O(1)}{x}$ and
\[
M(s)=\sum_{m\sim M}\frac{a_m}{m^s},\quad Z^\flat(s)=\sum_{n\asymp x/M}\frac{h(n)\chi(n)}{n^s}-\delta_\chi \frac{\varphi(q)}{q} H(1-s),\quad \delta_\chi=\begin{cases}1,\quad &\text{if $\chi=\chi_0$;}, \\ 0,&\text{otherwise.}\end{cases}
\]
Then for $T'\ll T^2$ we have
\[
\int_{T'}^{T'+T}|M(1/2+it)Z^\flat(1/2+it)|dt\ll qT^{1+o(1)}+qM^{1/2}T^{1/2+o(1)}.
\]
\end{lemma}


\begin{proof}
By Cauchy-Schwarz we have
\begin{align*}
 \int_{T'}^{T'+T} |M(1/2+it)Z^\flat(1/2+it)|dt&\ll \Bigl(\int_{T'}^{T'+T}|M(1/2+it)|^2dt\Bigr)^{1/2}\Bigl(\int_{T'}^{T'+T}|Z^\flat(1/2+it)|^2dt\Bigr)^{1/2}.
\end{align*}
By the mean value theorem for Dirichlet polynomials, the integral of $|M(1/2+it)|^2$ is $O(M+T)$. For the second factor, we have that, by Mellin inversion, for $\Re(s)=1/2$
\begin{align*}
Z^\flat(s)&=\frac{1}{2\pi i}\int_{2-i\infty}^{2+i\infty}L(s+s',\chi)H(s')ds'-\delta_\chi \frac{\varphi(q)}{q} H(1-s)\\
&=\frac{1}{2\pi i}\int_{-\infty}^\infty H(it)L(s+it,\chi)dt\\
&\ll \int_{-\infty}^\infty \frac{|L(s+it,\chi)|}{1+t^2}dt.
\end{align*}
Here we used the fact that $H(it)\ll 1/(1+t^2)$ from the bounds on the derivatives of $h$ and that the residue at $s'=1-s$ cancels out the $\delta_\chi \frac{\varphi(q)}{q} H(1-s)$ term. By the reflection principle (see, for example \cite[Equation (21)]{Perelli}), there exists coefficients $a^{(1)}_n$, $a^{(2)}_n\ll \tau(n)^{O(1)}$ such that for $t\in [T,2T]$ we have
\[ |L(1/2+it,\chi)|\ll \Bigl|\sum_{n\ll qT^{1/2}}\frac{a^{(1)}_n}{n^{1/2+it}}\Bigr|+\Bigl|\sum_{n\ll qT^{1/2}}\frac{a^{(2)}_n}{n^{1/2+it}}\Bigr|+1.\]
Applying this and the mean value theorem for Dirichlet polynomials, we find
\begin{align*}
\int_{T'}^{T'+T} |Z^\flat(1/2+it_1)|^2dt_1&\ll \int_{T'}^{T'+T}\int_{-\infty}^\infty \frac{|L(1/2+it_1+it_2,\chi)|^2}{1+t_2^2}dt_2dt_1\\
&\ll q\log{T}\int_{-\infty}^{\infty}\frac{T+|T'+t_2|^{1/2}}{1+t_2^2}dt_2\ll q (T+(T')^{1/2})(\log{T})^{O(1)}.
\end{align*}
Putting these bounds together then gives the result.
\end{proof}


\begin{lemma}[Type I sums]\label{lem:TypeI}
Let $0<\lambda<x^{1-2\theta+o(1)}$ and $a_m\ll \tau(m)^{O(1)}$ be a complex sequence supported on $m\le x^{\theta-\varepsilon}$. Let $g$ be a smooth function supported on $[x,x+x^\theta]$ with $|g^{(j)}(t)|\ll x^{-j\theta}$ for all $j\ge 0$, and let $g_1$ be a smooth function supported on $[x,x+h_1]$ with $h_1=x\exp(-\log^{1/2}{x})$ and $|g^{(j)}(t)|\ll h_1^{-j}$ for all $j\ge0$. Let $\chi$ be a Dirichlet character mod $q\ll \log^{O(1)}{x}$. Then
\[
\sum_{m n\in [x,x+x^\theta]}a_m\chi(m n) g(m n)\e(\lambda m n)=\delta_\chi \Bigl(\int_0^\infty g(t)\e(\lambda t)dt\Bigr) \Bigl( \frac{1}{\widehat{g_1}(0)} \sum_{\substack{mn \in [x,x+h_1] \\ (mn, q) = 1}} a_m g_1(mn)\Bigr) +O\Bigl(\frac{x^\theta}{\log^{B}{x}}\Bigr).
\]
for any $B \geq 1$.
\end{lemma}


\begin{proof}
We may clearly assume that $a_m$ is supported on $m\sim M<x^{\theta-\varepsilon}$. Let $G_\lambda(s)$ be the Mellin transform of $g(u)\e(\lambda u)$ given by
\[
G_\lambda(s)=\int_0^\infty u^{s-1}g(u)\e(\lambda u)du.
\]
Let $h$ be a smooth function supported on $[x/4M,4x/M]$ which is 1 on $[x/2M,2x/M]$ and satisfies $|h^{(j)}(t)|\ll (M/x)^j$ for $j\in\{0,1,2\}$. Let $H(s)$ be the Mellin transform of $h$. Thus, inserting an artificial factor $h(n)$ and using Mellin inversion, we find
\begin{align*}
\sum_{m n\in [x,x+x^\theta]}&a_m\chi(m n) g(m n)\e(\lambda m n)=\sum_{m n\in [x,x+x^\theta]}a_m h(n)\chi(m n) g(m n)\e(\lambda m n)\\
&=\frac{1}{2\pi i} \int_{2-i\infty}^{2+i\infty} G_\lambda(s)\Bigl(\sum_{m\sim M}\frac{a_m\chi(m)}{m^s}\Bigr)\Bigl(\sum_{n}\frac{h(n)\chi(n)}{n^s}\Bigr)ds\\
&=\frac{1}{2\pi i} \int_{-\infty}^{\infty} G_\lambda(1/2+it)\Bigl(\sum_{m\sim M}\frac{a_m\chi(m)}{m^{1/2+it}}\Bigr)\Bigl(\sum_{n}\frac{h(n)\chi(n)}{n^{1/2+it}}-\delta_\chi \frac{\varphi(q)}{q} H(1/2-it)\Bigr)dt\\
&\qquad +\frac{\delta_\chi}{2\pi i} \cdot \frac{\varphi(q)}{q} \sum_{m\sim M}\frac{a_m\chi_0(m)}{m}\int_0^\infty g(u)\e(\lambda u)\int_{2-i\infty}^{2+i\infty}\frac{u^{s-1}H(1-s)}{m^{s-1}}ds du.
\end{align*}
By Lemma~\ref{lem:Stationary}, we have $|G_\lambda(1/2+it)|\ll x^{-10}/t^2$ unless $|t+2\pi\lambda x|<x^{1-\theta+\varepsilon/2}$. Thus we may restrict the integral in the first term to an interval of length $O(x^{1-\theta+\varepsilon/2})$ inside $[0,x^{2-2\theta+\varepsilon}]$ at the cost of a negligible error. Since $G_\lambda(1/2+it)\ll x^{\theta-1/2}$, this remaining integral is $O(x^{\theta-\varepsilon/4+o(1)}{x})$ by Lemma~\ref{lem:TypeIIntegral}, and so the first term is $O(x^\theta\log^{-B}{x})$. The second term here is simply 
\[
\frac{\varphi(q)}{q} \delta_\chi\sum_{m\sim M, (m,q)=1}\frac{a_m}{m}\int_0^\infty g(u)\e(\lambda u)h(u/m)du=\delta_\chi\int_0^\infty g(u)\e(\lambda u)du\sum_{m\sim M,(m,q)=1}\frac{a_m}{m},
\]
by Mellin inversion for $h$ and the fact that $h(u/m)=1$ for $m\sim M$ and $u$ in the support of $g$. This shows that
\[
\sum_{m n\in [x,x+x^\theta]}a_m\chi(m n) g(m n)\e(\lambda m n)=\delta_\chi \frac{\varphi(q)}{q} \Bigl(\int_0^\infty g(t)\e(\lambda t)dt\Bigr)\sum_{(m,q)=1}\frac{a_m}{m}+O\Bigl(\frac{x^\theta}{\log^{B}{x}}\Bigr).
\]
Now applying the same argument to the long interval $[x,x+h_1]$ with the smooth function $g_1$ in place of $g$, and taking $\chi= \chi_0$ and $\lambda = 0$, we obtain
\[ 
\sum_{m n\in [x,x+h_1]}\chi_0(mn) a_m g_1(m n)= \frac{\varphi(q)}{q} \widehat{g_1}(0) \sum_{(m,q)=1}\frac{a_m}{m}+O\Bigl(\frac{h_1}{\log^{B}{x}}\Bigr).
\]
The lemma follows by comparing the two equations above.
\end{proof}


\begin{lemma}[Type II integrals]\label{lem:TypeIIIntegral}
Let $a_j,b_j,c_j\in\mathbb{C}$ satisfy $|a_j|$, $|b_j|$, $|c_j|\ll \tau(j)^{O(1)}$ for all $j$, and let
\[
M(s)=\sum_{n\sim M}\frac{a_n}{n^s},\qquad N(s)=\sum_{n\sim N}\frac{b_n}{n^s},\qquad Z(s)=\sum_{n\sim Z}\frac{c_n}{n^s},
\]
where $M N Z\asymp x$. Let $\beta_1,\beta_2$ be such that $M=x^{\beta_1}$, $N=x^{\beta_2}$ and $\beta_1,\beta_2$ satisfy
\begin{equation}
\beta_1+\beta_2>4/5-4\varepsilon,\qquad |\beta_1-\beta_2|<1/10+2\varepsilon.\label{eq:Ranges}
\end{equation}
Let $T=x^{1-\theta+\varepsilon}$, with $\theta\ge 11/20+2\varepsilon$, and assume that for some $T'$ we have that
\begin{equation}
\sup_{t\in [T',T'+T]}\Bigl|Z\Bigl(\frac{1}{2}+it\Bigr)\Bigr|\ll \frac{Z^{1/2}}{\log^{B}{x}}\label{eq:PrimeFactored}
\end{equation}
for any $B \geq 1$.
Then we have
\[
\int_{T'}^{T'+T}\Bigr|M\Bigl(\frac{1}{2}+it\Bigr)N\Bigl(\frac{1}{2}+it\Bigr)Z\Bigl(\frac{1}{2}+it\Bigr)\Bigl|dt\ll \frac{x^{1/2}}{\log^{B}{x}}
\]
for any $B \geq 1$.
\end{lemma}


\begin{proof}
This follows from \cite[Lemma 7.3]{Harman-book} with $\theta=11/20+\varepsilon$ and $g=2$ after absorbing a factor $n^{i T'}$ into $a_n,b_n,c_n$ to change the range of integration to $[0,T]$. The proof works by considering sets $\mathcal{S}(\sigma_1,\sigma_2,\sigma_3)$ of well-spaced points $t_1,\dots,t_V$ in $[0,T]$ for which $|M(1/2+it_j)|$, $|N(1/2+it_j)|$ and $|Z(1/2+it_j)|$ are of size $M^{\sigma_1-1/2}$, $N^{\sigma_2-1/2}$ and $Z^{\sigma_3-1/2}$ respectively, and applying the mean-value and large-value bound
\[
\#\mathcal{S}(\sigma_1,\sigma_2,\sigma_3)\ll (\log{x})^{O(1)}\Bigl(M^{2-2\sigma_1}+T\min(M^{1-2\sigma_1},M^{4-6\sigma_1})\Bigr)
\]
and similar bound with $(M, \sigma_1)$ replaced by $(N, \sigma_2)$ or $(Z^h,\sigma_3)$ for positive integers $h$,
along with the property~\eqref{eq:PrimeFactored} that $Z^{\sigma_3}\ll Z(\log{x})^{-B}$ for any $B \geq 1$ if $\mathcal{S}(\sigma_1,\sigma_2,\sigma_3)\ne 0$. Together these can be used to show that whenever~\eqref{eq:Ranges} holds, we have $\#\mathcal{S}(\sigma_1,\sigma_2,\sigma_3)\ll x M^{-\sigma_1}N^{-\sigma_2}Z^{-\sigma_3}(\log{x})^{-B}$ for any $B \geq 1$.
\end{proof}


\begin{lemma}[Type II sums]\label{lem:TypeII}
Let $c_j^{(II)}(n)$ be as in Lemma~\ref{lem:SieveDecomposition}, let $0<\lambda<x^{1-2\theta+o(1)}$ and let $\chi$ be a Dirichlet character mod $q\ll \log^{O(1)}{x}$. Let $g$ be a smooth function supported on $[x,x+x^\theta]$ with $|g^{(j)}(t)|\ll x^{-j\theta}$ for all $j\ge 0$, and let $g_1$ be a smooth function supported on $[x,x+h_1]$ with $h_1=x\exp(-\log^{1/2}{x})$ and $|g^{(j)}(t)|\ll h_1^{-j}$ for all $j\ge0$. Then we have
\[
\sum_{n\in [x,x+x^\theta]}c_j^{(II)}(n)\e(\lambda n)g(n)\chi(n)=\Bigl(\int_0^\infty g(t)\e(\lambda t)dt\Bigr)\Bigl(\frac{\delta_\chi}{\widehat{g}_1(0)}\sum_{n\in [x,x+h_1]}c_j^{(II)}(n)g_1(n)\Bigr)+O\Bigl(\frac{x^\theta}{\log^{B}{x}}\Bigr)
\]
for any $B \geq 1$.
\end{lemma}


\begin{proof}
Let $G_\lambda(s)$ be the Mellin transform of $g(u)\e(\lambda u)$ as in the proof of Lemma~\ref{lem:TypeI}. Then we have
\begin{align}
\sum_{n\in[x,x+x^\theta]}c_j^{(II)}(n)\e(\lambda n)g(n)\chi(n)=\frac{1}{2\pi i}\int_{1/2-i\infty}^{1/2+i\infty}\Bigl(\sum_{n\asymp x}\frac{c_j^{(II)}(n)\chi(n)}{n^s}\Bigr)G_\lambda(s)ds.\label{eq:MellinExpansion}
\end{align}
Since the sum in parentheses is $O(x^{1/2+\varepsilon})$, by Lemma~\ref{lem:Stationary} we have that the integral with $|t+2\pi\lambda x|>T=x^{1-\theta+\varepsilon}$ (where $t = \Im(s)$) contributes a total of $O(1)$, and so we may restrict to $|t+2\pi\lambda x|<T$. Recalling the shape of $c_j^{(II)}(n)$ from Lemma~\ref{lem:SieveDecomposition}, we have
\begin{align*}
\sum_{n\asymp x}\frac{c_j^{(II)}(n)\chi(n)}{n^s}=\int_{-x^{1+\varepsilon}}^{x^{1+\varepsilon}} \int_{-x^{1+\varepsilon}}^{x^{1+\varepsilon}} \int_{-x^{1+\varepsilon}}^{x^{1+\varepsilon}} E(s;\mathbf{t})F(s;\mathbf{t})P(s;\mathbf{t})K(\mathbf{t})dt_1 dt_2 dt_3,
\end{align*}
where $\mathbf{t}=(t_1,t_2,t_3)$ and 
\[
E(s;\mathbf{t})=\sum_{n\sim E}\frac{e_j(n;\mathbf{t})\chi(n)}{n^s}\,\qquad F(s;\mathbf{t})=\sum_{n\sim F}\frac{f_j(n;\mathbf{t})\chi(n)}{n^s},\qquad P(s;\mathbf{t})=\sum_{p\sim P}\frac{\chi(p)}{p^{s+i(t_1+ t_2 +t_3)}}.
\]
Writing $u=t_1+t_2+t_3$ we obtain
\begin{equation}\label{eq:TypeIIsum-1} 
\int_{\mathcal{C}}\Bigl|\sum_{n\asymp x}\frac{c_j^{(II)}(n)\chi(n)}{n^s}\Bigr||ds| \ll \sup_{\mathbf{v}\in\R^3}\int_{-3x^{1+\ee}}^{3x^{1+\ee}} \frac{\log^{O(1)}x}{1+|u|} \int_{\mathcal{C}}|E(s;\mathbf{v}) F(s;\mathbf{v})P(s+iu;\mathbf{0})| |ds|du, 
\end{equation}
for any contour $\mathcal{C}$, where we have used the estimate
\begin{align*} 
& \int_{-x^{1+\ee}}^{x^{1+\ee}} \int_{-x^{1+\ee}}^{x^{1+\ee}} \Bigl|K(t_1,t_2,u-t_1-t_2)\Bigr| dt_1 dt_2 \\
\ll & \log^{O(1)}x \int_{-x^{1+\ee}}^{x^{1+\ee}} \int_{-x^{1+\ee}}^{x^{1+\ee}} (1+|t_1|)^{-1} (1+|t_2|)^{-1} (1+|u-t_1-t_2|)^{-1} dt_1\cdots dt_2 \\
\ll & \frac{\log^{O(1)}x}{1+|u|}. 
\end{align*}
If $\chi$ is non-principal then, by the zero-free region for $L(s,\chi)$, we have
\begin{equation}
|P(s+iu; \mathbf{0)}|\ll \frac{P^{1/2}}{\log^{B}{x}}\label{eq:ZeroRegionBound}
\end{equation}
for any $B \geq 1$ and all $s,u$ occurring in the integral above. Thus, Lemma~\ref{lem:TypeIIIntegral} shows that
\[
\int_{-2\pi\lambda x-T}^{-2\pi\lambda x+T}\Bigl|\sum_{n\asymp x}\frac{c_j^{(II)}(n)\chi(n)}{n^{1/2+it}}\Bigr|dt\ll \frac{x^{1/2}}{\log^{B}{x}}.
\]
Recalling that $G_\lambda(1/2+it)\ll x^{\theta-1/2}$, we find that for $\chi\ne\chi_0$ this shows
\[
\sum_{n\in[x,x+x^\theta]}c_j^{(II)}(n)\chi(n)g(n)\e(\lambda n)\ll \frac{x^\theta}{\log^B x}.
\]

If $\chi=\chi_0$ is the principal character, then we can replace $\chi$ by 1 since $c^{(II)}_j(n)$ are supported on $(n,P(\omega))=1$. We separate the contribution from $t\in[-T_0,T_0]$, where $T_0=\exp(\log^{1/2}{x})$. From~\eqref{eq:TypeIIsum-1} we see that
\begin{align*}
\int_{1/2+iT_0}^{1/2+i(T_0+T)}\Bigl|\sum_{n\asymp x}\frac{c_j^{(II)}(n)}{n^{s}}\Bigr||ds| \ll \sup_{\mathbf{v}}\int_{-3x^{1+\varepsilon}}^{3x^{1+\varepsilon}}\frac{\log^{O(1)}{x}}{1+|u|}\int_{1/2+iT_0}^{1/2+i(T_0+T)}|E(s;\mathbf{v})F(s;\mathbf{v})P(s+iu;\mathbf{0})||ds| du.
\end{align*}
We have that the bound~\eqref{eq:ZeroRegionBound} holds whenever $|\Im(s)+u|>T_0^{1/2}$ by the zero free region for $\zeta(s)$. Thus we split the integral above into two regions according to whether $|t+u|<T_0^{1/2}$ or $|t+u|>T_0^{1/2}$, where $t=\Im(s)$. As before, when $|t+u|>T_0^{1/2}$  Lemma~\ref{lem:TypeIIIntegral} shows that the total contribution to the integral is $O(x^{1/2}\log^{-B}{x})$. In the remaining case, we see that $|t+u|<T_0^{1/2}$ implies that $|u|\gg |t|>T_0$. Thus, first integrating over $u$ and using the trivial bound $P(s+iu;\mathbf{0})\ll P^{1/2}$, we find that this part contributes a total
\[
\ll P^{1/2}T_0^{1/2+\varepsilon}\sup_{\mathbf{v}}\int_{T_0}^{T_0+T}|E(1/2+it;\mathbf{v})F(1/2+it;\mathbf{v})|\frac{dt}{t}.
\]
Splitting into dyadic ranges and applying the mean value theorem, we see that this
\[
\ll P^{1/2}T_0^{1/2+2\varepsilon}\sup_{T\ge T_0}\frac{(E+T)^{1/2}(F+T)^{1/2}}{T}\ll \frac{(EFP)^{1/2}}{T_0^{1/2-2\varepsilon}}\ll \frac{x^{1/2}}{\log^{B}{x}}
\]
for any $B \geq 1$.
Thus, since $G_\lambda(1/2+it)\ll x^{\theta-1/2}$, we see that the total contribution to~\eqref{eq:MellinExpansion} from $|\Im(s)|>T_0$ is $O(x^\theta\log^{-2A}{x})$. Hence
\begin{align*}
\sum_{n\in[x,x+x^\theta]}c_j^{(II)}(n)\e(\lambda n)g(n)=\frac{1}{2\pi i}\int_{1/2-iT_0}^{1/2+iT_0}\Bigl(\sum_{n\asymp x}\frac{c_j^{(II)}(n)}{n^s}\Bigr)G_\lambda(s)ds+O\Bigl(\frac{x^{\theta}}{\log^{B}{x}}\Bigr).
\end{align*}

For $|t|<T_0$ we use the Taylor expansion of $u^{it-1/2}$ about $u=x$ to obtain
\[
G_\lambda(1/2+it)=\int_x^{x+x^\theta} g(u)u^{it-1/2}\e(\lambda u)du=x^{it-1/2}\int_0^\infty g(u)\e(\lambda u)du+O(x^{2\theta-3/2+\varepsilon}).
\]
Thus
\[
\sum_{n\in[x,x+x^\theta]}c_j^{(II)}(n)\e(\lambda n)g(n)\chi_0(n)=\Bigl(\frac{1}{2\pi i}\int_{1/2-iT_0}^{1/2+iT_0}\Bigl(\sum_{n\asymp x}\frac{c_j^{(II)}(n)}{n^s}\Bigr)x^{s-1}ds\Bigr)\Bigl(\int_0^\infty g(u)\e(\lambda u)du\Bigr)+O\Bigl(\frac{x^{\theta}}{\log^{B}{x}}\Bigr).
\]
By applying the same argument to the interval $[x,x+h_1]$ with a smooth function $g_1$ in place of $g$, and taking $\lambda=0$, we see that
\[
\frac{1}{2\pi i}\int_{1/2-i T_0}^{1/2+i T_0}\Bigl(\sum_{n\asymp x}\frac{c_j^{(II)}(n)}{n^s}\Bigr)x^{s-1}ds=\frac{1}{\widehat{g}_1(0)}\sum_{n\in [x,x+h_1]}c_j^{(II)}(n)g_1(n)+O\Bigl(\frac{1}{\log^{B}{x}}\Bigr).
\]
This gives the result.
\end{proof}


Combining Lemma~\ref{lem:TypeI} and Lemma~\ref{lem:TypeII} along with our decomposition from Lemma~\ref{lem:SieveDecomposition} finally allows us to complete the proof of Proposition~\ref{prop:SAz}.


\begin{lemma}[Propositon~\ref{prop:SAz} for $\gamma$ in major arcs]
Let $g$ be a smooth function supported on $[x,x+x^\theta]$ with $|g^{(j)}(t)|\ll x^{-j\theta}$ for all $j\ge0$, and let $g_1$ be a smooth function supported on $[x,x+h_1]$ with $h_1=x\exp(-\log^{1/2}{x})$ and $|g_1^{(j)}(t)|\ll h_1^{-j}$ for all $j\ge0$. Let $(c,d)=1$ with $d\ll\log^{O(1)}{x}$ and $\gamma=a/q+\lambda$ with $(a,q)=1$, $q\ll \log^{5A}{x}$ and $\lambda\ll x^{1-2\theta+o(1)}$. Then we have
\[
\sum_{\substack{n\in [x,x+x^\theta] \\ n\equiv c\pmod{d} }}\rho^+(n)g(n)\e(n\gamma)=\frac{[q,d]}{\varphi([q,d])}\Bigl(\frac{1}{\widehat{g_1}(0)}\sum_{n\in [x,x+h_1]}g_1(n)\rho^+(n)\Bigr)\Bigl(\sum_{\substack{n\in [x,x+x^\theta]\\ n\equiv c\pmod{d} \\ (n,q)=1}}g(n)\e(\gamma n)\Bigr)+O\Bigl(\frac{x^\theta}{\log^{B}{x}}\Bigr)
\]
for any $B \geq 1$.
\end{lemma}


\begin{proof}
By Lemma~\ref{lem:SieveDecomposition} we obtain sequences $c_j^{(I)}(n)$ and $c_j^{(II)}(n)$ for $j\ll \log^{O(1)}{x}$ such that
\[
\sum_{\substack{n\in [x,x+x^{\theta}]\\ n\equiv c\pmod{d} }}\rho^+(n)g(n)\e(\gamma n)=\sum_{j\le \log^{O(1)}{x} }\sum_{\substack{n\in [x,x+x^\theta]\\ n\equiv c\pmod{d} \\ (n, dq) = 1}}(c_j^{(I)}(n)+c_j^{(II)}(n))g(n)\e(\gamma n)+O\Bigl(\frac{x^\theta}{\log^B{x}}\Bigr).
\]
Here we have been able to add the summation condition $(n, dq) = 1$ to the right hand side since the left hand side is supported on $(n, dq) = 1$. Since the same decomposition holds also with $x^\theta$ replaced by $h_1$ and $g$ replaced by $g_1$, it suffices to prove the desired equality with $\rho^+$ replaced by $c_j^{(I)}$ and $c_j^{(II)}$ and with additional summation condition $(n, dq) = 1$. 

For the type-I coefficients, we have
\begin{equation}\label{eq:minor-1} 
\sum_{\substack{n \in [x,x+x^{\theta}] \\ n \equiv c\pmod{d} \\ (n, dq) = 1}} c_j^{(I)}(n) g(n) \e(\gamma n) = \sum_{\substack{r\pmod{[d,q]} \\ r\equiv c\pmod{d} \\ (r, dq) = 1}} \e(ar/q) \sum_{\substack{n \in [x,x+x^{\theta}] \\ n\equiv r\pmod{[d,q]}}} c_j^{(I)}(n) g(n) \e(\lambda n). 
\end{equation}
Since $(r, dq) = 1$, by orthogonality of characters we obtain
\[
\sum_{\substack{n \in [x,x+x^{\theta}] \\ n\equiv r\pmod{[d,q]}}} c_j^{(I)}(n) g(n) \e(\lambda n) = \frac{1}{\varphi([d,q])} \sum_{\chi\pmod{[d,q]}} \overline{\chi(r)} \sum_{n\in [x,x+x^{\theta}]} c_j^{(I)}(n) \chi(n) g(n) \e(\lambda n). 
\]
The inner sum over $n$ can be evaluated using Lemma~\ref{lem:TypeI} to be
\[ 
\delta_{\chi} \Bigl( \int_0^{\infty} g(t)\e(\lambda t)dt \Bigr) \Bigl( \frac{1}{\widehat{g_1}(0)} \sum_{\substack{n \in [x,x+h_1] \\ (n, dq) = 1}} c_j^{(I)}(n) g_1(n) \Bigr) + O\Bigl( \frac{x^{\theta}}{\log^{B}x} \Bigr).
\]
Hence
\[
\begin{split}
&\sum_{\substack{n \in [x,x+x^{\theta}] \\ n \equiv c\pmod{d}}} c_j^{(I)}(n) g(n) \e(\gamma n) \\
&= \frac{1}{\varphi([d, q])} \sum_{\substack{r\pmod{[d,q]} \\ r\equiv c\pmod{d} \\ (r, dq) = 1}} \e(ar/q) \Bigl( \int_0^{\infty} g(t)\e(\lambda t)dt \Bigr) \Bigl( \frac{1}{\widehat{g_1}(0)} \sum_{\substack{n \in [x,x+h_1] \\ (n, dq) = 1}} c_j^{(I)}(n) g_1(n) \Bigr) + O\Bigl( \frac{x^{\theta}}{\log^{B}x} \Bigr).
\end{split}
\]
Here
\[
\begin{split}
\sum_{\substack{r\pmod{[d,q]} \\ r\equiv c\pmod{d} \\ (r, dq) = 1}} \e(ar/q) \Bigl( \int_0^{\infty} g(t)\e(\lambda t)dt \Bigr) &= \sum_{\substack{r\pmod{[d,q]} \\ r\equiv c\pmod{d} \\ (r, q) = 1}} \e(ar/q) \cdot [d, q] \sum_{\substack{n \in [x, x+x^\theta] \\ n \equiv r \pmod{[d, q]}}} g(n)\e(\lambda n) + O(q [d, q] |\lambda|  x^{\theta}) \\
&= [d, q] \sum_{\substack{n \in [x, x+x^\theta] \\ n \equiv c \pmod{d} \\ (n, q) = 1}} g(n)\e(\gamma n) + O\left(\frac{x^\theta}{(\log x)^B}\right)
\end{split}
\]
and the claim follows for type I coefficients. The type II coefficients can be treated completely analogously, replacing the application of Lemma~\ref{lem:TypeI} by an application of Lemma~\ref{lem:TypeII}.
\end{proof}


\bibliographystyle{plain}
\bibliography{biblio}

\end{document}